\documentclass[]{amsart}
\usepackage{mathrsfs}
\usepackage{color}
\usepackage{amsmath}
\usepackage{amsfonts}
\usepackage{color}
\usepackage{amssymb}

 \newtheorem{Theorem}{Theorem}[section]
 \newtheorem{Corollary}[Theorem]{Corollary}
 \newtheorem{Lemma}[Theorem]{Lemma}
 \newtheorem{Proposition}[Theorem]{Proposition}

\newtheorem{Question}[Theorem]{Question}

 \newtheorem{Remark}[Theorem]{Remark}

 \numberwithin{equation}{section}

\begin{document}

\title[An Optimal Support Function]
 {An Optimal Support Function related to the strong openness conjecture}

\author{Qi'an Guan and Zheng Yuan}

\address{Qi'an Guan, School of Mathematical Sciences,
Peking University, Beijing, 100871, China.}

\address{Zheng Yuan, School of Mathematical Sciences,
Peking University, Beijing, 100871, China.}

\email{guanqian@math.pku.edu.cn \\  zyuan@pku.edu.cn}


\thanks{}


\date{\today}

\dedicatory{}

\commby{}

\maketitle
\begin{abstract}
	In the present article, we obtain an optimal support function
of weighted $L^2$
integrations on superlevel sets of psh weights, which implies the strong openness property of multiplier ideal sheaves.
\end{abstract}

\section{Introduction}

The notion of  multiplier ideal sheaves plays an important role in complex geometry and algebraic geometry,
which has been widely discussed (see e.g. \cite{tian87,Nadel90,siu96,DEL00,D-K01,demailly-note2000,D-P03,Laz04,siu05,siu09,demailly2010}).

Let $D$ be a domain in $\mathbb C^n$, and let $\psi$ be a negative plurisubharmonic function (see \cite{demailly-book,Si,siu74}) on $D$.
Recall that the multiplier ideal sheaf $\mathcal I(\psi)$ was defined as the sheaf of germs of holomorphic functions $f$ such that
$|f|^{2}e^{-\psi}$ is locally integrable (see \cite{demailly2010}). It is known that $\mathcal I(\psi)$ is a coherent analytic sheaf (see \cite{Nadel90} and \cite{demailly2010}).
In \cite{demailly2010} (see also \cite{demailly-note2000}), Demailly posed the strong openness conjecture for multiplier ideal sheaves:
$\mathcal I(\psi)=\mathcal I_+(\psi):=\cup_{\epsilon>0}\mathcal I((1+\epsilon)\psi)$.
When $\mathcal I(\psi)=\mathcal{O}$, the strong openness conjecture was called the openness conjecture.
In \cite{berndtsson13}, Berndtsson proved the openness conjecture (Favre-Jonsson proved 2-dim in \cite{FM05j}).
In \cite{GZopen-c} (see also \cite{Lempert14} and \cite{Hiep14}), Guan-Zhou proved the strong openness conjecture,
which is called the strong openness property (Jonsson-Mustata proved 2-dim in \cite{JM12}).
After that, Guan-Zhou \cite{GZ-lelong1}
gave a characterization of the multiplier ideal sheaves with weights of Lelong number
one by  using the strong openness property of multiplier ideal sheaves. Recently, Xu \cite{xu} completed the algebraic approach to the openness conjecture,
which was conjectured by Jonsson-Mustata \cite{JM12}.

Based on some modifications of the method in \cite{GZopen-effect} (see also \cite{guan-zhou13p} and \cite{guan-zhou13ap}),  a support function of weighted $L^2$ integrations on the superlevel sets of the weight was obtained in \cite{GLN-support function}, which implies the strong openness property.
Then it is natural to ask:

\begin{Question}
\label{Q:main_optimal}
Can one obtain an optimal version of the above support function?
\end{Question}

Let $z_0$ be a point in a pseudoconvex domain $D\subset \mathbb C^n$, and let $F$ be a holomorphic function on $D$. Let $D_t=\{z\in D:\psi\geq -t\}$ the superlevel set of weight $\psi$, and let   $C_{F,\psi,t}(z_0)$ be the infimum of $\int_{D_t}|F_1|^2$ for all $F_1\in \mathcal O(D)$ satisfying that $(F_1-F,z_0)\in \mathcal I(\psi)_{z_0}$. When $C_{F,\psi,t}(z_0)=0$ or $\infty$, we set $\frac{\int_{D_t}|F|^2e^{-\psi}	}{C_{F,\psi,t}(z_0)}=+\infty$.

In the present article, by doing some modifications of the method in \cite{guan_sharp}, we
obtain the following optimal uniform support function of $\frac{\int_{D_t}|F|^2e^{-\psi}	}{C_{F,\psi,t}(z_0)}$,
which gives a positive answer to Question \ref{Q:main_optimal}.

\begin{Theorem}
\label{thm: support function}
Let $D$ be a pseudoconvex domain in $\mathbb{C}^n$.
Let $F$ be a holomorphic function, and let $\psi$ be a negative plurisubharmonic function on $D$.
Assume that $\int_{\{\psi<-l\}}|F|^2<+\infty$ holds for any $l>0$.	
Then the inequality
\begin{equation}\label{eq: main}
\frac{\int_{D_t}|F|^2e^{-\psi}	}{C_{F,\psi,t}(z_0)}\geqslant\frac{t}{1-e^{-t}}
\end{equation}
holds for any $t\in(0,+\infty)$,
where $\frac{t}{1-e^{-t}}$ is the optimal support function.
\end{Theorem}

The following remark illustrates the optimality of the support function in inequality \eqref{eq: main}.

\begin{Remark}
\label{r:optimality}
Take $D=\Delta\subset\mathbb{C}$, $z_{0}=\mathnormal{0}$ the origin of $\mathbb{C}$, $F\equiv1$ and $\psi=2\log{|z|}$.
It is clear that $\int_D|F|^2<+\infty$. Note that $C_{F,\psi,t}(z_0)=\pi(1-e^{-t})$ and $\int_{D_t}|F|^2e^{-\psi}=t\pi$.
Then $\frac{\int_{D_t}|F|^2e^{-\psi}	}{C_{F,\psi,t}(z_0)}=\frac{t}{1-e^{-t}}$, which shows the optimality of the support function $\frac{t}{1-e^{-t}}$.
\end{Remark}

Theorem \ref{thm: support function} implies the following strong openness property. We present some details in Section \ref{p of soc}.

\begin{Corollary}(see \cite{GZopen-c})
	\label{soc}
     $\mathcal{I}(\psi)=\mathcal{I}_+(\psi)$.
\end{Corollary}

\section{Preparations}

In this section, we will do some preparations.

\subsection{$\overline{\partial}$-equation with $L^2$ estimates}
We call a positive smooth function $c$ on $(0,+\infty)$ in class $P$, if the following three statements hold

(1) $\int_0^{+\infty}c(t)e^{-t}dt<+\infty$;

(2) $c(t)e^{-t}$ is decreasing with respect to $t$;

(3) for any $a>0$, $c(t)$ has a positive lower bound on $(a,+\infty)$.

\

We shall prove Lemma \ref{l:5} by using the following Lemma, whose various forms already appear in \cite{guan-zhou13p,guan-zhou13ap,guan_sharp} etc.:

\begin{Lemma}
\label{l: estimates}
Let $B\in(0,+\infty)$ and $t_0>0$ be arbitrarily given. Let $D$ be a pseudoconvex domain in $\mathbb{C}^n$. Let $\psi$ be a negative plurisubharmonic function on $D$, and $\varphi$ be a plurisubharmonic function on $D$. Let $F$ be a holomorphic function on $\{\psi<-t_0\}$ such that
\begin{displaymath}
\int_{K\cap\{\psi<-t_0\}}|F|^2<+\infty
\end{displaymath}
for any compact subset $K$ of $D$ and
\begin{equation}
\label{eq:l1}
\int_D\frac{1}{B}\mathbb{I}_{\{-t_0-B<\psi<-t_0\}}|F|^2e^{-\varphi}\leq C <+\infty,	
\end{equation}
where $C\in\mathbb{R}$ is a positive constant. Then there exists a holomorphic function $\widetilde{F}$ on $D$  such that
\begin{equation}
\label{eq: estimates}	
\int_D|\widetilde{F}-(1-b_{t_0,B}(\psi))F|^2e^{-\varphi+v_{t_0,B}(\psi)}c(-v_{t_0,B}(\psi))\leq C\int_0^{t_0+B}c(t)e^{-t}dt,
\end{equation}
where $b_{t_0,B}(t)=\int_{-\infty}^t\frac{1}{B}\mathbb{I}_{\{-t_0-B<s<-t_0\}}ds$, $v_{t_0,B}(t)=\int_0^tb_{t_0,B}(s)ds$, and \mbox{$c(t)\in P$}.
\end{Lemma}

For the sake of the convenience of the reader, we give the details of the proof of Lemma \ref{l: estimates} in Section \ref{sec:Lemma_L2}.

\begin{Remark}
\label{r}

Note that $v_{t_0,B}(t)\geq\max\{t,-(t_{0}+B)\}$, $c(t)e^{-t}$ is decreasing with respect to $t$,
and $b_{t_0,B}(t)=0$ for any $t\leq -(t_{0}+B)$.
Replacing $\varphi$ by $\psi$ and assuming $\psi(z_0)=-\infty$, it follows from inequality \eqref{eq: estimates} that
	\begin{displaymath}
		\begin{split}
			&\int_{\{\psi<-t_0-B\}}|\widetilde{F}-F|^2e^{-\psi-(t_0+B)}c(t_0+B)	
			\\\leq	&\int_{\{\psi<-t_0-B\}}|\widetilde{F}-(1-b_{t_0,B}(\psi))F|^2e^{-\varphi+v_{t_0,B}(\psi)}c(-v_{t_0,B}(\psi))
			\\\leq &\int_{D}|\widetilde{F}-(1-b_{t_0,B}(\psi))F|^2e^{-\varphi+v_{t_0,B}(\psi)}c(-v_{t_0,B}(\psi))
			\\\leq &C\int_0^{t_0+B}c(t)e^{-t}dt
			\\ <&+\infty,	
			\end{split}
	\end{displaymath}
	which implies that $(\widetilde{F}-F,z_0)\in\mathcal I(\psi)_{z_0}$.
\end{Remark}

\subsection{Some properties of $G_{F,\psi,c}(t)$}

Let $\psi$ be a negative plurisubharmonic function defined on a pseudoconvex domain $D\subset \mathbb C^n$,
and let $F$ be a holomorphic function on $\{z\in D:\psi<-t\}$, and $z_0\in\{z\in D:\psi<-t\}$.
Denote that
\begin{displaymath}
	G_{F,\psi,c}(t):=\inf\left\{\int_{\{\psi<-t\}}|F_1|^2c(-\psi):(F_1-F,z_0)\in\mathcal I(\psi)_{z_0} \& F_1\in\mathcal O(\{\psi<-t\})\right\},
\end{displaymath} where $c\in P$.

We recall the closedness of submodules, which can be referred to  \cite{G-R} (Chapter 2, \S3, 3. Closedness of Submodules, page 45).

\begin{Lemma}(see \cite{G-R}) 
\label{closedness}
Let $N$ a submodule of $\mathcal O_{\mathbb C^n,o}^q$, $1\leq q<+ \infty$, let $f_j\in\mathcal O_{\mathbb C^n,o}(U)^q$ be a sequence of $q-$tuples holomorphic in an open neighborhood $U$ of the origin. Assume that the $f_j$ converge uniformly in $U$ towards to a $q-$tuples $f\in\mathcal{O}_{\mathbb C^n,o}^q$, assume furthermore that all germs $(f_{j},o)$ belong to $N$. Then $(f,o)\in N$.
	
\end{Lemma}

The closedness of submodules will be used in the following discussion.

\begin{Lemma}
\label{l:2}
Let $F$ be a holomorphic function on $D$, then $(F,z_0)\in\mathcal{I}(\psi)_{z_0}$ if and only if  $G_{F,\psi,c}(0)=0$, where $(F,z_0)$ denotes the germ of $F$ at $z_0$.
\end{Lemma}

\begin{proof}
It is clear that $(F,z_0)\in\mathcal{I}(\psi)_{z_0}$ implies that $G_{F,\psi,c}(0)=0$.

If $ G_{F,\psi,c}(0)=0$, then there exist holomorphic functions $\{f_j\}_{j\in \mathbb{N}}$ on $D$ such that $\lim_{j\rightarrow +\infty}\int_D|f_j|^2c(-\psi)=0$ and $(f_j-F,z_0)\in \mathcal{I}(\psi)_{z_0}$ for any $j$. As $c(-\psi)$ has positive lower bound on any compact subset of $D$, then there exists a subsequence of  $\{f_j\}_{j\in \mathbb{N}}$  denoted by $\{f_{j_k}\}_{k\in \mathbb{N}}$
 compactly convergent to 0.
Then we know $f_{j_k}-F$ is compactly convergent to $-F$.
It follows from Lemma \ref{closedness} that $(F,z_0)\in\mathcal{I}(\psi)_{z_0}$. This proves Lemma \ref{l:2}. 
 \end{proof}

 \begin{Lemma}
 	\label{l:3}
Assume that $G_{F,\psi,c}(t)<+\infty$, then there exists a unique holomorphic function $F_t$ on $\{\psi<-t\}$ satisfying $(F_t-F,z_0)\in\mathcal{I}(\psi)_{z_0}$ and $\int_{\{\psi<-t\}}|F_t|^2c(-\psi)=G_{F,\psi,c}(t)$.
Furthermore, for any holomorphic function $\hat
 	{F}$ on $\{\psi<-t\}$ satisfying $(\hat{F}-F,z_0)\in\mathcal{I}(\psi)_{z_0}$ and $\int_{\{\psi<-t\}}|\hat{F}|^2c(-\psi)<+\infty$, we have the following equality
    \begin{equation}
    	\label{eq:l31}
    	\int_{\{\psi<-t\}}|F_t|^2c(-\psi)+\int_{\{\psi<-t\}}|\hat{F}-F_t|^2c(-\psi)=\int_{\{\psi<-t\}}|\hat{F}|^2c(-\psi).
    \end{equation}
 	
 \end{Lemma}

 \begin{proof}
 Firstly, we prove the existence of  $F_t$. As $G_{F,\psi,c}(t)<+\infty$, then there exist holomorphic functions $\{f_j\}_{j\in \mathbb{N}}$ on $\{\psi<-t\}$ such that
 $\lim_{j\rightarrow+\infty}\int_{\{\psi<-t\}}|f_j|^2c(-\psi)=G_{F,\psi,c}(t)$ and $(f_j-F,z_0)\in\mathcal{I}(\psi)_{z_0}$. Then there exists a subsequence of  $\{f_j\}_{j\in \mathbb{N}}$  denoted by $\{f_{j_k}\}_{k\in \mathbb{N}}$ compactly convergence to a holomorphic function $f$ on $\{\psi<-t\}$ satisfying $\int_K|f|^2c(-\psi)\leq G_{F,\psi,c}(t)$ for any compact subset $K\subset\{\psi<-t\}$,
 which implies that $\int_{\{\psi<-t\}}|f|^2c(-\psi)\leq G_{F,\psi,c}(t)$ by the monotone convergence theorem.
 Note that Lemma \ref{closedness} implies that $(f-F,z_0)\in\mathcal{I}(\psi)_{z_0}$. Then we obtain the existence of $F_t(=f)$.

 Secondly, we prove the uniqueness of $F_t$ by contradiction: if not, there exist two different holomorphic functions $f_1$ and $f_2$ on $\{\psi<-t\}$ satisfying $\int_{\{\psi<-t\}}|f_1|^2c(-\psi)=\int_{\{\psi<-t\}}|f_2|^2c(-\psi)=G_{F,\psi,c}(t)$, $(f_1-F,z_0)\in\mathcal{I}(\psi)_{z_0}$, and $(f_2-F,z_0)\in\mathcal{I}(\psi)_{z_0}$. Note that
\begin{equation}
\label{eq:l32}
\begin{split}	\int_{\{\psi<-t\}}\left\vert\frac{f_1+f_2}{2}\right\vert^2c(-\psi)+\int_{\{\psi<-t\}}\left\vert\frac{f_1-f_2}{2}\right\vert^2c(-\psi)
	\\=\frac{\int_{\{\psi<-t\}}|f_1|^2c(-\psi)+\int_{\{\psi<-t\}}|f_2|^2c(-\psi)	}{2}=G_{F,\psi,c}(t),
\end{split}
\end{equation}
then we obtain that
\begin{displaymath}
	\int_{\{\psi<-t\}}\left\vert\frac{f_1+f_2}{2}\right\vert^2c(-\psi)< G_{F,\psi,c}(t) ,
\end{displaymath}
and $(\frac{f_1+f_2}{2}-F,z_0)\in\mathcal{I}(\psi)_{z_0}$, which contradicts the definition of $G_{F,\psi,c}(t)$.

Finally, we prove the equality \eqref{eq:l31}. For any holomorphic function $f$ on $\{\psi<-t\}$ satisfying $(f,z_0)\in\mathcal{I}(\psi)_{z_0}$ and $\int_{\{\psi<-t\}}|f|^2c(-\psi)<+\infty$, 
 it is clear that for any complex number $\alpha$, $F_t+\alpha f$ satisfies 

\begin{displaymath}	
((F_t+\alpha f)-F,z_0)=((F_t-F)+\alpha f,z_0)\in\mathcal{I}(\psi)_{z_0}
\end{displaymath}
and the definition of $G_{F,\psi,t}(t)$ implies that
\begin{displaymath} \int_{\{\psi<-t\}}|F_t|^2c(-\psi)\leq\int_{\{\psi<-t\}}|F_t+\alpha f|^2c(-\psi)<+\infty.
\end{displaymath}
Note that
 \begin{displaymath} \int_{\{\psi<-t\}}|F_t+\alpha f|^2c(-\psi)-\int_{\{\psi<-t\}}|F_t|^2c(-\psi)\geq0	
 \end{displaymath}
implies
\begin{displaymath}
	\mathrm{Re}\int_{\{\psi<-t\}}F_t\overline{f}c(-\psi)=0,
\end{displaymath}
then we obtain that
\begin{displaymath}
\int_{\{\psi<-t\}}|F_t+f|^2c(-\psi)=\int_{\{\psi<-t\}}|F_t|^2c(-\psi)+\int_{\{\psi<-t\}}|f|^2c(-\psi).
\end{displaymath}
Choosing $f=\hat{F}-F_t$, we obtain equality \eqref{eq:l31}.\end{proof}

 \begin{Lemma}
 	\label{l:4}
 	Let $F$ be a holomorphic function on $D$ and $\psi(z_0)=-\infty$. Assume that $G_{F,\psi,c}(0)<+\infty$. Then $G_{F,\psi,c}(t)$ is decreasing with respect to $t\in[0,+\infty)$	 such that $\lim_{t\rightarrow t_0+}G_{F,\psi,c}(t)= G_{F,\psi,c}(t_0)$ for any $t_0\in[0,+\infty)$, $\lim_{t\rightarrow t_0-}G_{F,\psi,c}(t)\geq G_{F,\psi,c}(t_0)$ for any $t_0\in(0,+\infty)$ and $\lim_{t\rightarrow +\infty}G_{F,\psi,c}(t)=0$. 	Especially, $G_{F,\psi,c}(t)$ is lower semi-continuous on $[0,+\infty)$.
 	\end{Lemma}
 	
\begin{proof}
By the definition of $G_{F,\psi,c}(t)$, it is clear that $G_{F,\psi,c}(t)$ is decreasing on $[0,+\infty)$, $\lim_{t\rightarrow t_0-}G_{F,\psi,c}(t)\geq G_{F,\psi,c}(t_0)$ for any $t\in(0,+\infty)$ and $\lim_{t\rightarrow +\infty}G_{F,\psi,c}(t)=0$. It suffices to prove $\lim_{t\rightarrow t_0+}G_{F,\psi,c}(t)= G_{F,\psi,c}(t_0)$. We prove it by contradiction: if not, then $\lim_{t\rightarrow t_0+}G_{F,\psi,c}(t)< G_{F,\psi,c}(t_0)$ for some $t_0\in[0,+\infty)$.

By Lemma \ref{l:3},  there exists a unique holomorphic function $F_t$ on $\{\psi<-t\}$ satisfying that $(F_t-F,z_0)\in\mathcal{I}(\psi)_{z_0}$ and $\int_{\{\psi<-t\}}|F_t|^2c(-\psi)=G_{F,\psi,c}(t)$. Note that $G_{F,\psi,c}(t)$ is decreasing implies that $\int_{\{\psi<-t\}}|F_t|^2c(-\psi)\leq\lim_{t\rightarrow t_0+}G_{F,\psi,c}(t)$ for any $t>t_0$. As $c(-\psi)$ has positive lower bound on any compact subset of $\{\psi<-t_0\}$,  then there exists a sequence of $\{F_{t_j}\}$ ($t_j\rightarrow t_0$, as $j\rightarrow+\infty$) uniformly convergent on any compact subset of $\{\psi<-t_0\}$.
Let $\hat{F}_{t_0}:=\lim_{j\rightarrow+\infty}F_{t_j}$, which is a holomorphic function on $\{\psi<-t_0\}$ and $(\hat{F}_{t_0}-F,z_0)\in\mathcal{I}(\psi)_{z_0}$. Then it follows the decreasing property of $G_{F,\psi,c}(t)$ that 
\begin{equation}
	\nonumber
	\int_{K}|\hat{F}_{t_0}|^2c(-\psi)=\lim_{j\rightarrow+\infty}\int_{K}|F_{t_j}|^2c(-\psi)\leq\lim_{j\rightarrow+\infty}G_{F,\psi,c}(t_j)<G_{F,\psi,c}(t_0)-\epsilon
\end{equation} for any compact subset of $\{\psi<-t_0\}$ and $\epsilon>0$ is independent of $K$. It follows from the monotone convergence theorem that 
$	\int_{\{\psi<-t_0\}}|\hat{F}_{t_0}|^2c(-\psi)<G_{F,\psi,c}(t_0)$,
 which contradicts the definition of $G_{F,\psi,c}(t_0)$.
  \end{proof}

The following Lemma will be used to prove Proposition \ref{p}.

\begin{Lemma}
	\label{l:5}
	Let $F$ be a holomorphic function on $D$ and $\psi(z_0)=-\infty$. Assume that $G_{F,\psi,c}(0)<+\infty$. Then for any $t_0\in(0,+\infty)$, we have
	\begin{equation}
		\label{eq:l51}
		G_{F,\psi,c}(0)-G_{F,\psi,c}(t_0)\leq\frac{\int_0^{t_0}c(t)e^{-t}dt}{c(t_0)e^{-t_0}}\liminf_{B\rightarrow0+}\frac{G_{F,\psi,c}(t_0)-G_{F,\psi,c}(t_0+B)}{B}.		
	\end{equation}
	\end{Lemma}

\begin{proof}
	By Lemma \ref{l:3}, there exists a holomorphic function $F_{t_0}$ on $\{\psi<-t_0\}$, such that $(F_{t_0}-F,z_0)\in\mathcal{I}(\psi)_{z_0}$ and $\int_{\{\psi<-t_0\}}|F_{t_0}|^2c(-\psi)=G_{F,\psi,c}(t_0)$.
	
	It suffices to consider that $\liminf_{B\rightarrow0+}\frac{G_{F,\psi,c}(t_0)-G_{F,\psi,c}(t_0+B)}{B}\in[0,+\infty)$, because of the decreasing property of $G_{F,\psi,c}(t)$. Then there exists a sequence  $\{B_j\} $ ($B_j\rightarrow0+$, as $j\rightarrow+\infty$) such that
	\begin{equation}
		\label{eq:l52}
		\lim_{j\rightarrow+\infty}\frac{G_{F,\psi,c}(t_0)-G_{F,\psi,c}(t_0+B_j)}{B_j}=\liminf_{B\rightarrow0+}\frac{G_{F,\psi,c}(t_0)-G_{F,\psi,c}(t_0+B)}{B}	
		\end{equation}
		and $\left\{\frac{G_{F,\psi,c}(t_0)-G_{F,\psi,c}(t_0+B_j)}{B_j}\right\}_{j\in{\mathbb{N}}}$ is bounded.

		Lemma \ref{l: estimates} and Remark \ref{r} ($\varphi\sim\psi $ and $F\sim F_{t_0}$, where $\sim$ means the former replaced by the latter) show that for any $B_j$, there exists holomorphic function $\widetilde{F}_j$ on $D$ such that $(\widetilde{F}_j-F_{t_0},z_0)\in\mathcal I(\psi)_{z_0}$ and
		\begin{equation}
			\label{eq:l53}
			\begin{split}
				&\int_D|\widetilde F_j-(1-b_{t_0,B_j}(\psi))F_{t_0}|^2c(-v_{t_0,B_j}(\psi))e^{v_{t_0,B_j}(\psi)-\psi}
			\\\leq &\int_0^{t_0+B_j}c(t)e^{-t}dt\cdot\int_D\frac{1}{B_j}\mathbb{I}_{\{-t_0-B_j<\psi<-t_0\}}|F_{t_0}|^2e^{-\psi}
			\\\leq &\frac{e^{t_0+B_j}\int_0^{t_0+B_j}c(t)e^{-t}dt}{\inf_{t\in(t_0,t_0+B_j)}c(t)}\cdot\int_D\frac{1}{B_j}\mathbb I_{\{-t_0-B_j<\psi<-t_0\}}|F_{t_0}|^2c(-\psi)
			\\	\leq &\frac{e^{t_0+B_j}\int_0^{t_0+B_j}c(t)e^{-t}dt}{\inf_{t\in(t_0,t_0+B_j)}c(t)}\cdot\frac{G_{F,\psi,c}(t_0)-G_{F,\psi,c}(t_0+B_j)}{B_j}.
			\end{split}
		\end{equation}
		As $-t\leq{v_{t_0,B_j}(-t)}$ for any $t\geq0$, the decreasing property of $c(t)e^{-t}$ shows that $c(t)e^{-t}\leq c(-v_{t_0,B_j}(-t))e^{v_{t_0,B_j}(-t)}$ for any $t\in(0,+\infty)$, which implies that
		\begin{equation}
		\label{eq:l54}
			c(-\psi)\leq c(-v_{t_0,B_j}(\psi))e^{v_{t_0,B_j}(\psi)-\psi}.		
			\end{equation}
			 Then combining inequality \eqref{eq:l53} and inequality \eqref{eq:l54}, we obtain that
			\begin{equation}
				\label{eq:l56}
				\begin{split}
					&\int_D|\widetilde F_j-(1-b_{t_0,B_j}(\psi))F_{t_0}|^2c(-\psi)                    \\	\leq &\frac{e^{t_0+B_j}\int_0^{t_0+B_j}c(t)e^{-t}dt}{\inf_{t\in(t_0,t_0+B_j)}c(t)}\cdot\frac{G_{F,\psi,c}(t_0)-G_{F,\psi,c}(t_0+B_j)}{B_j}.
				\end{split}
			\end{equation}
			
			Firstly, we will prove that  $\int_D|\widetilde F_j|^2c(-\psi)$ is bounded with respect to $j$.
			
			Note that
			\begin{displaymath}
				\begin{split}
				&\left(\int_D|\widetilde F_j|^2c(-\psi)\right)^{\frac{1}{2}}- \left(\int_D|(1-b_{t_0,B_j}(\psi))F_{t_0}|^2c(-\psi) \right)^{\frac{1}{2}}
				\\\leq & \left(\int_D|\widetilde F_j-(1-b_{t_0,B_j}(\psi))F_{t_0}|^2c(-\psi) \right)^{\frac{1}{2}}	,			
				\end{split}			
				\end{displaymath}
			then we obtain that
			\begin{displaymath}
			\begin{split}
				 \left(\int_D|\widetilde F_j|^2c(-\psi) \right)^{\frac{1}{2}}	\leq	 & \left(\int_D|(1-b_{t_0,B_j}(\psi))F_{t_0}|^2c(-\psi) \right)^{\frac{1}{2}}
				\\&+ 	\left(\frac{e^{t_0+B_j}\int_0^{t_0+B_j}c(t)e^{-t}dt}{\inf_{t\in(t_0,t_0+B_j)}c(t)}\cdot\frac{G_{F,\psi,c}(t_0)-G_{F,\psi,c}(t_0+B_j)}{B_j} \right)^{\frac{1}{2}}	.				
			\end{split}				
			\end{displaymath}
			Since $\left\{\frac{G_{F,\psi,c}(t_0)-G_{F,\psi,c}(t_0+B_j)}{B_j}\right\}_{j\in{\mathbb{N}}}$ is bounded and
			\begin{displaymath}
				\int_D|(1-b_{t_0,B_j}(\psi))F_{t_0}|^2c(-\psi)\leq \int_{\{\psi<-t_0\}}|F_{t_0}|^2c(-\psi)<+\infty,			
				\end{displaymath}
				then we obtain that $\int_D|\widetilde F_j|^2c(-\psi)$ is bounded with respect to $j$.
				
				Secondly, we will prove the main result.
				
				It follows from $b_{t_0,B_j}(\psi)=1$ on $\{\psi\geq-t_0\}$ that 	
				\begin{equation}
					\label{eq:l57}
					\begin{split}
					&\int_D|\widetilde F_j-(1-b_{t_0,B_j}(\psi))F_{t_0}|^2c(-\psi)
					\\=&\int_{\{\psi\geq-t_0\}}|\widetilde F_j|^2c(-\psi)+	\int_{\{\psi<-t_0\}}|\widetilde F_j-(1-b_{t_0,B_j}(\psi))F_{t_0}|^2c(-\psi).					
					\end{split}
				\end{equation}
			
			Denote that $||\cdot||_{2}:= \left(\int_{\{\psi<-t_{0}\}}|\cdot|^{2}c(-\psi) \right)^{\frac{1}{2}}$.
It is clear that
\begin{equation}
\label{eq:l58}
\begin{split}
&||\tilde{F}_{j}-(1-b_{t_0,B_j}(\psi))F_{t_{0}}||_{2}^{2}
\\\geq&\left(||\tilde{F}_{j}-F_{t_{0}}||_{2}-||b_{t_0,B_{j}}(\psi)F_{t_{0}}||_{2}\right)^{2}
\\\geq&||\tilde{F}_{j}-F_{t_{0}}||_{2}^{2}-2||\tilde{F}_{j}-F_{t_{0}}||_{2}||b_{t_0,B_{j}}(\psi)F_{t_{0}}||_{2}
\\\geq&
||\tilde{F}_{j}-F_{t_{0}}||_{2}^{2}-2||\tilde{F}_{j}-F_{t_{0}}||_{2}
 \left(\int_{\{-t_{0}-B_{j}<\psi<-t_{0}\}}|F_{t_{0}}|^{2}c(-\psi) \right)^{\frac{1}{2}},
\end{split}
\end{equation}
where the last inequality follows from $0\leq b_{t_{0},B_{j}}(\psi)\leq 1$ and $b_{t_{0},B_{j}}(\psi)=0$ on $\{\psi\leq -t_{0}-B_{j}\}$.

Combining equality \eqref{eq:l57}, inequality \eqref{eq:l58} and equality \eqref{eq:l31},
we obtain that
\begin{equation}
\label{eq:l59}
\begin{split}
&\int_{D}|\tilde{F}_{j}-(1-b_{t_{0},B_{j}}(\psi))F_{t_{0}}|^{2}c(-\psi)
\\=&\int_{\{\psi\geq-t_{0}\}}|\tilde{F}_{j}|^{2}c(-\psi)+||\tilde{F}_{j}-(1-b_{t_0,B_j}(\psi))F_{t_{0}}||_{2}^{2}
\\\geq&
\int_{\{\psi\geq-t_{0}\}}|\tilde{F}_{j}|^{2}c(-\psi)+||\tilde{F}_{j}-F_{t_{0}}||_{2}^{2}-2||\tilde{F}_{j}-F_{t_{0}}||_{2}||b_{t_0,B_{j}}(\psi)F_{t_{0}}||_{2}
\\\geq&
\int_{\{\psi\geq-t_{0}\}}|\tilde{F}_{j}|^{2}c(-\psi)+
||\tilde{F}_{j}||_{2}^{2}-||F_{t_{0}}||_{2}^{2}
\\&-2||\tilde{F}_{j}-F_{t_{0}}||_{2}
 \left(\int_{\{-t_{0}-B_{j}<\psi<-t_{0}\}}|F_{t_{0}}|^{2}c(-\psi) \right)^{\frac{1}{2}}
\\=&
\int_{D}|\tilde{F}_{j}|^{2}c(-\psi)-||F_{t_{0}}||_{2}^{2}
\\&-2||\tilde{F}_{j}-F_{t_{0}}||_{2}
 \left(\int_{\{-t_{0}-B_{j}<\psi<-t_{0}\}}|F_{t_{0}}|^{2}c(-\psi) \right)^{\frac{1}{2}}.
\end{split}
\end{equation}

It follows from equality \eqref{eq:l31} that
\begin{equation}
\label{eq:l510}
\begin{split}
||\tilde{F}_{j}-F_{t_{0}}||_{2}
= \left(||\tilde{F}_{j}||_{2}^{2}-||F_{t_{0}}||_{2}^{2} \right)^{\frac{1}{2}}
\leq||\tilde{F}_{j}||_{2}\leq \left(\int_{D}|\tilde{F}_{j}|^{2}c(-\psi) \right)^{\frac{1}{2}}.
\end{split}
\end{equation}
Since $\int_{D}|\tilde{F}_{j}|^{2}c(-\psi)$ is bounded with respect to $j$,
inequality \eqref{eq:l510} implies that
$ \left(\int_{\{\psi<-t_{0}\}}|\tilde{F}_{j}-F_{t_{0}}|^{2}c(-\psi) \right)^{\frac{1}{2}}$ is bounded with respect to $j$.
Using the dominated convergence theorem and $\int_{\{\psi<-t_{0}\}}|F_{t_{0}}|^{2}c(-\psi)<+\infty$,
we obtain that
\begin{displaymath}
\lim_{j\to +\infty}||\tilde{F}_{j}-F_{t_{0}}||_{2}
 \left(\int_{\{-t_{0}-B_{j}<\psi <-t_{0}\}}|F_{t_{0}}|^{2}c(-\psi) \right)^{\frac{1}{2}}=0.	
\end{displaymath}
Combining with inequality \eqref{eq:l59},
we obtain
\begin{equation}
\label{eq:l511}
\begin{split}
&\liminf_{j\to+\infty}\int_{D}|\tilde{F}_{j}-(1-b_{t_{0},B_{j}}(\psi))F_{t_{0}}|^{2}c(-\psi)
\\\geq&\liminf_{j\to+\infty}\int_{D}|\tilde{F}_{j}|^{2}c(-\psi)-||F_{t_{0}}||_{2}^{2}.
\end{split}
\end{equation}

Using inequality \eqref{eq:l56} and inequality \eqref{eq:l511}, we obtain
\begin{equation}
\label{eq:l512}
\begin{split}
&\frac{\int_{0}^{t_{0}}c(t)e^{-t}dt}{c(t_{0})e^{-t_{0}}}
\lim_{j\to+\infty}\left(\frac{G_{F,\psi,c}(t_{0})-G_{F,\psi,c}(t_{0}+B_{j})}{B_{j}}\right)
\\=&\lim_{j\to+\infty}\frac{e^{t_{0}+B_{j}}\int_{0}^{t_{0}+B_{j}}c(t)e^{-t}dt}{\inf_{t\in(t_{0},t_{0}+B_{j})}c(t)}
\left(\frac{G_{F,\psi,c}(t_{0})-G_{F,\psi,c}(t_{0}+B_{j})}{B_{j}}\right)
\\\geq&\liminf_{j\to+\infty}\int_{D}|\tilde{F}_{j}-(1-b_{t_{0},B_{j}}(\psi))F_{t_{0}}|^{2}c(-\psi)
\\\geq&\liminf_{j\to+\infty}\int_{D}|\tilde{F}_{j}|^{2}c(-\psi)-||F_{t_{0}}||_{2}^{2}
\\\geq& G_{F,\psi,c}(0)-G_{F,\psi,c}(t_{0}).
\end{split}
\end{equation}
This proves Lemma \ref{l:5}.			
\end{proof}

The following Lemma will be used to prove Theorem \ref{thm: support function}.
\begin{Lemma}
	\label{l:6}
	Let $F$ be a holomorphic function on pseudoconvex domain $D$, and let $\psi$ be a negative plurisubharmonic function on $D$. Assume that $\int_{\{\psi\geq-t\}}|F|^2e^{-\psi}<+\infty$ for a given positive number $t$. Then
	\begin{equation}
		\label{eq:l61}
		\int_{\{\psi\geq-t\}}|F|^2e^{-\psi}=	\int_{-\infty}^t\left(\int_{\{-t\leq\psi<-l\}}|F|^2\right)e^ldl.	
	\end{equation}	
\end{Lemma}

\begin{proof}
	It is clear that the lemma directly follows from the basic formula
	\begin{equation}
	\label{eq:l62}
		\int_Xfd\mu=\int_0^{+\infty}\mu(\{x\in X:f(x)>l\})dl
	\end{equation}
  for a measurable function $f:X\rightarrow[0,+\infty)$ with $X=\{\psi>-t\}$, $f=e^{-\psi}$, and $d\mu=|F|^2dV_{2n}$, where $dV_{2n}$ is the Lebesgue measure on $\mathbb C^n$.

  Next we will prove the equality \eqref{eq:l62}.
  \begin{displaymath}
  	\begin{split}
  		\int_Xfd\mu &=\int_X\left(\int_0^{f(x)}dl\right)d\mu
  		\\&=\int_X\left(\int_0^{+\infty} \theta_f(l,x) dl\right)d\mu
  		\\&=\int_0^{+\infty}\left(\int_X\theta_f(l,x)d\mu\right)dl
  		\\&=\int_0^{+\infty}\mu(\{x\in X:f(x)>l\})dl,  		\end{split}
  \end{displaymath}
  where
    $\theta_f(l,x) =
  \left\{ \begin{array}{cc}
  1 & \textrm {if $l<f(x)$}\\
  0 & \textrm{if $l\geq f(x)$}  	
  \end{array} \right.$
  is a function defined on $\mathbb R\times X$. Then the equality \eqref{eq:l62} has thus been proved.
  \end{proof}

\section{Proof of Theorem \ref{thm: support function}}
Firstly, we prove the following proposition.
\begin{Proposition}
\label{p}
	Let $F$ be a holomorphic function on $D$, $z_0\in D$, and $\psi$ be a negative plurisubharmonic function on $D$. Then the inequality
	\begin{equation}
		\label{eq:p1}
		\int_{\{\psi<-t\}}|F|^2
c(-\psi)\geq \frac{\int_t^{+\infty}c(l)e^{-l}dl}{\int_0^{+\infty}c(l)e^{-l}dl}G_{F,\psi,c}(0)	
    \end{equation}
    holds for any $t\geq0$, which is sharp.

     Especially, if $G_{F,\psi,c}(0)=+\infty$, then $\int_{\{\psi<-t\}}|F|^2
c(-\psi)=+\infty$ for  any $t\geq0$.
\end{Proposition}

\begin{Remark}
Let $D=\Delta\subset\mathbb C$ be the unit disc, $z_0=o$ the origin of $\mathbb C$, $\psi =2\log{|z|}$, $F\equiv1$, and $c\equiv1$. It is clear that 
$$e^{-t}\pi=\int_{\{\psi<-t\}}|F|^2
c(-\psi)\geq\frac{\int_t^{+\infty}c(l)e^{-l}dl}{\int_0^{+\infty}c(l)e^{-l}dl}G_{F,\psi,c}(0)	=e^{-t}\pi,$$
 which is sharp.	
\end{Remark}

\begin{proof}[Proof of Proposition 3.1]
When $\psi(z_0)>-\infty$, it is clear that $G_{F,\psi,c}(0)=0$. Then the inequality \eqref{eq:p1} holds. Next, we will only consider the case  $\psi(z_0)=-\infty$.

We prove Proposition \ref{p} in two steps, i.e. the case $G_{F,\psi,c}(0)<+\infty$ and the case $G_{F,\psi,c}(0)=+\infty$.

Step 1. We prove the case $G_{F,\psi,c}(0)<+\infty$. As $\int_{\{\psi<-t\}}|F|^2c(-\psi)\geq G_{F,\psi,c}(t)$ for any $t\in[0,+\infty)$, then it suffices to prove that $G_{F,\psi,c}(t)\geq\frac{\int_t^{+\infty}c(l)e^{-l}dl}{\int_0^{+\infty}c(l)e^{-l}dl}G_{F,\psi,c}(0)$ for any $t\in[0,+\infty)$.

 Let $H(t):=G_{F,\psi,c}(t)-\frac{\int_t^{+\infty}c(l)e^{-l}dl}{\int_0^{+\infty}c(l)e^{-l}dl}G_{F,\psi,c}(0)$.We prove $H(t)\geq0$ by contradiction: if not, then there exists $t$ such that $H(t)<0$.

 Note that $G_{F,\psi,c}(t)\in[0,G_{F,\psi,c}(0)]$ is bounded on $[0,+\infty)$, then $H(t)$ is also bounded on $[0,+\infty)$, which implies that $\inf_{[0,+\infty)}H(t)$ is finite.

By Lemma \ref{l:4}, it is clear that $\lim_{t\rightarrow0+}H(t)=H(0)=0$ and $\lim_{t\rightarrow+\infty}H(t)=0$. Then it follows from $\inf_{[0,+\infty)}H(t)<0$ that there exists a closed interval $[a,b]\subset(0,+\infty)$ such that $\inf_{[a,b]}H(t)=\inf_{[0,+\infty)}H(t)$. Since $G_{F,\psi,c}$ is lower semi-continuous, then $H(t)$ is also lower semi-continuous, which implies that there exists $t_0\in[a,b]$ such that $H(t_0)=\inf_{[0,+\infty)}H(t)<0$.

As $H(t_0)=\inf_{[0,+\infty)}H(t)$, then it follows that $\liminf_{B\rightarrow0+}\frac{H(t_0)-H(t_0+B)}{B}\leq0$. Combining with $H(t_0)<0$,then we obtain that
\begin{equation}
\label{eq:p2}
H(t_0)+\frac{\int_0^{t_0}c(t)e^{-t}dt}{c(t_0)e^{-t_0}}\liminf_{B\rightarrow0+}\frac{H(t_0)-H(t_0+B)}{B}<0	.
\end{equation}

Note that
\begin{displaymath}
	\begin{split}
		&H(t_0)+\frac{\int_0^{t_0}c(t)e^{-t}dt}{c(t_0)e^{-t_0}}\liminf_{B\rightarrow0+}\frac{H(t_0)-H(t_0+B)}{B}
		\\=&G_{F,\psi,c}(t_0)-\frac{\int_{t_0}^{+\infty}c(t)e^{-t}dt}{\int_0^{+\infty}c(t)e^{-t}dt}G_{F,\psi,c}(0)
		     \\&+\frac{\int_0^{t_0}c(t)e^{-t}dt}{c(t_0)e^{-t_0}}\liminf_{B\rightarrow0+}\frac{G_{F,\psi,c}(t_0)-G_{F,\psi,c}(t_0+B)-\frac{\int_{t_0}^{t_0+B}c(t)e^{-t}dt}{\int_{0}^{+\infty}c(t)e^{-t}dt}G_{F,\psi,c}(0)}{B}
		     \\=&G_{F,\psi,c}(t_0)+\frac{\int_0^{t_0}c(t)e^{-t}dt}{c(t_0)e^{-t_0}}\liminf_{B\rightarrow0+}\frac{G_{F,\psi,c}(t_0)-G_{F,\psi,c}(t_0+B)}{B}	
		         \\&-G_{F,\psi,c}(0)\left(\frac{\int_{t_0}^{+\infty}c(t)e^{-t}dt}{\int_0^{+\infty}c(t)e^{-t}dt}+\frac{\int_{0}^{t_0}c(t)e^{-t}dt}{\int_0^{+\infty}c(t)e^{-t}dt}\right)	
		      \\=&G_{F,\psi,c}(t_0)-	G_{F,\psi,c}(0)+\frac{\int_0^{t_0}c(t)e^{-t}dt}{c(t_0)e^{-t_0}}\liminf_{B\rightarrow0+}\frac{G_{F,\psi,c}(t_0)-G_{F,\psi,c}(t_0+B)}{B},		
		      \end{split}
\end{displaymath}
then it follows from Lemma \ref{l:5} that $$H(t_0)+\frac{\int_0^{t_0}c(t)e^{-t}dt}{c(t_0)e^{-t_0}}\liminf_{B\rightarrow0+}\frac{H(t_0)-H(t_0+B)}{B}\geq0,$$
which contradicts inequality \eqref{eq:p2}. Then the case $G_{F,\psi,c}(0)<+\infty$ has thus been proved.

Step 2. We prove the case $G_{F,\psi,c}(0)=+\infty$ by contradiction: if not, then integral $\int_{\{\psi <-t_0\}}|F|^2c(-\psi)$ is finite for some $t_0>0$ (when $t_0=0$, $\int_{\{\psi <-t_0\}}|F|^2c(-\psi)\geq G_{F,\psi,c}(0)=+\infty$). It follows from Lemma \ref{l: estimates} and Remark \ref{r} that there exists a holomorphic function $\widetilde F$ on $D$ satisfying $$(\widetilde{F}-F,z_0)\in\mathcal I(\psi)_{z_0}$$
and
\begin{equation}
	\label{eq:p3}
	\begin{split}
	&\int_D|\widetilde F-(1-b_{t_0,B}(\psi))F|^2c(-\psi)	\\
	\leq &\int_D|\widetilde F-(1-b_{t_0,B}(\psi))F|^2c(-v_{t_0,B}(\psi))e^{v_{t_0,B}(\psi)-\psi}	\\
	\leq &\int_0^{t_0+B}c(t)e^{-t}dt\int_D\frac{1}{B}\mathbb I_{\{-t_0-B<\psi<-t_0\}}|F|^2e^{-\psi}.
	\end{split}
	\end{equation}	
	Note that
	\begin{displaymath}
		\begin{split}
			&\left\vert \left(\int_D|\widetilde F|^2c(-\psi) \right)^{\frac{1}{2}}- \left(\int_D|(1-b_{t_0,B}(\psi)) F|^2c(-\psi) \right)^{\frac{1}{2}}\right\vert\\
		\leq & \left(\int_D|\widetilde F-(1-b_{t_0,B}(\psi))F|^2c(-\psi) \right)^{\frac{1}{2}},
		\end{split}
	\end{displaymath}
 then it follows from inequality \eqref{eq:p3} that
 \begin{equation}
 	\begin{split}
 		 \left(\int_D|\widetilde F|^2c(-\psi) \right)^{\frac{1}{2}}\leq & \left(\int_D|(1-b_{t_0,B}(\psi)) F|^2c(-\psi) \right)^{\frac{1}{2}}\\
 		&+ \left(\int_0^{t_0+B}c(t)e^{-t}dt\int_D\frac{1}{B}\mathbb I_{\{-t_0-B<\psi<-t_0\}}|F|^2e^{-\psi} \right)^{\frac{1}{2}}. 		
  	\end{split}
 \end{equation}
 Since $b_{t_0,B}(\psi)=1$ on $\{\psi\geq t_0\}$, $0\leq b_{t_0,B}(\psi)\leq1$, $\int_{\{\psi<-t_0\}}|F|^2c(-\psi)<+\infty$, and $c(-\psi)$ has positive lower bound on $\{\psi<-t_0\}$, then we obtain that $$ \left(\int_D|(1-b_{t_0,B}(\psi)) F|^2c(-\psi) \right)^{\frac{1}{2}}<+\infty$$ and$$ \left(\int_0^{t_0+B}c(t)e^{-t}dt\int_D\frac{1}{B}\mathbb I_{\{-t_0-B<\psi<-t_0\}}|F|^2e^{-\psi} \right)^{\frac{1}{2}}<+\infty,$$
 which implies that $$\int_D|\widetilde F|^2c(-\psi)<+\infty.$$
 Then we obtain $G_{F,\psi,c}(0)\leq \int_D|\widetilde F|^2c(-\psi)<+\infty$, which contradicts $G_{F,\psi,c}(0)=+\infty$. The case $G_{F,\psi,c}(0)=+\infty$ has been proved.
 \end{proof}

In the following part, we will prove Theorem \ref{thm: support function} by using Proposition \ref{p} and Lemma \ref{l:6}.

Firstly, we will construct a family of functions $\{c_t^n(x)\}_{n\in \mathbb{N}}\subset P$, where $t$ is the given positive number in Theorem \ref{thm: support function}.

Let $f(x)=\left\{ \begin{array}{cc}
	e^{-\frac{1}{1-(x-1)^2}} & \textrm{if $|x-1|<1$}\\
	0 & \textrm{if $|x-1|\geq1$}
    \end{array} \right. $ be a real function defined on $\mathbb R$.
    It is clear that $f(x)\in C_0^\infty(\mathbb R)$ and $f(x)\geq0$ for any $x\in \mathbb R$. Then let $g_n(x)=\frac{n}{(n+1)d}\int_0^{nx}f(s)ds$, where $d=\int_{\mathbb R}f(s)ds$. It follows that $g_n(x)$ is increasing with respect to $x$, $g_n(x)\leq g_{n+1}(x)$ for any $n\in \mathbb{N}$ and $x\in \mathbb R$, and $\lim_{n\rightarrow+\infty}g_n(x)=\mathbb I_{\{s\in \mathbb R: s>0\}}(x)$ for all $x\in \mathbb R$.

    Now we construct $\{c_t^n(x)\}_{n\in \mathbb{N}}$ by setting $c_t^n(x)=1-g_n(x-t)$. It follows from the properties of $\{g_n(x)\}_{n\in \mathbb{N}}$ that $c_t^n(x)$ is decreasing with respect to $x$, $c_t^n(x)\geq c_t^{n+1}(x)$ for any $n\in \mathbb{N}$ and $x\in \mathbb R$, and $\lim_{n\rightarrow+\infty}c_t^n(x)=\mathbb I_{\{s\in \mathbb R: s\leq t\}}(x)$ for all $x\in \mathbb R$. Note that $c_t^n(x)\in[\frac{1}{n+1},1]$ on $x\in(0,+\infty)$, then $c_t^n(x)\in P$ for any $n\in \mathbb{N}$.

    Secondly, we will prove
    \begin{equation}
    \label{eq:proof4}
    	\int_{\{-t\leq\psi<-l\}}|F|^2\geq\frac{e^{-l}-e^{-t}}{1-e^{-t}}C_{F,\psi,t}(z_0)
    	 \end{equation}
    for any $l\in [0,t)$.

    For the case $l=0$, inequality \eqref{eq:proof4} holds by the definition of $C_{F,\psi,t}(z_0)$, then we only consider the case $l>0$.

    It follows from Proposition \ref{p} that
    \begin{equation}
    	\label{eq:proof1}
    	\int_{\{\psi<-l\}}|F|^2
c_t^n(-\psi)\geq \frac{\int_l^{+\infty}c_t^n(s)e^{-s}ds}{\int_0^{+\infty}c_t^n(s)e^{-s}ds}G_{F,\psi,c_t^n}(0).
	    \end{equation}
	    By $\int_{\{\psi<-l\}}|F|^2<+\infty$ and the properties of $\{c_t^n\}_{n\in \mathbb{N}}$, we obtain that
	    \begin{equation}
	    	\label{eq:proof2}
	    	\lim_{n\rightarrow+\infty}\int_{\{\psi<-l\}}|F|^2
c_t^n(-\psi)=\int_{\{-t\leq\psi<-l\}}|F|^2.	    \end{equation}
As $c_t^n(x)\geq\mathbb I_{\{s\in\mathbb R:s\leq t\}}(x)$ for any $x>0$ and $n\in \mathbb{N}$, then it follows from the definition of $G_{F,\psi,c}$ and $C_{F,\psi ,t}(z_0)$ that
\begin{equation}
	\label{eq:proof3}
	G_{F,\psi,c_t^n}(0)\geq C_{F,\psi ,t}(z_0).
\end{equation}
Combining inequality \eqref{eq:proof1}, equality \eqref{eq:proof2}, and inequality \eqref{eq:proof3}, we obtain that
\begin{displaymath}
	\begin{split}
		\int_{\{-t\leq\psi<-l\}}|F|^2 &=	\lim_{n\rightarrow+\infty}\int_{\{\psi<-l\}}|F|^2
c_t^n(-\psi)\\
&\geq \lim_{n\rightarrow+\infty} \frac{\int_l^{+\infty}c_t^n(s)e^{-s}ds}{\int_0^{+\infty}c_t^n(s)e^{-s}ds}C_{F,\psi ,t}(z_0)\\
&=\frac{e^{-l}-e^{-t}}{1-e^{-t}}C_{F,\psi ,t}(z_0).
	\end{split}	
\end{displaymath}
Then inequality \eqref{eq:proof4} has been proved.

Finally, we will finish the proof of Theorem \ref{thm: support function}.

By Lemma \ref{l:6} and inequality \eqref{eq:proof4}, we obtain that
\begin{displaymath}
	\begin{split}
		\int_{D_t}|F|^2
e^{-\psi}&=\int_{-\infty}^t\left(\int_{\{-t\leq\psi<-l\}}|F|^2\right)e^ldl\\
&=\int_{-\infty}^0\left(\int_{\{-t\leq\psi<-l\}}|F|^2\right)e^ldl+\int_{0}^t\left(\int_{\{-t\leq\psi<-l\}}|F|^2\right)e^ldl\\
&\geq C_{F,\psi,t}(z_0)\left(\int_{-\infty}^0e^ldl+\int_0^t\frac{1-e^{l-t}}{1-e^{-t}}dl\right)\\
&=\frac{t}{1-e^{-t}}C_{F,\psi,t}(z_0).	\end{split}
\end{displaymath}
Then Theorem \ref{thm: support function} has thus been proved.

\section{Theorem \ref{thm: support function} implies Corollary \ref{soc}}
\label{p of soc}

It suffices to prove that  $\mathcal{I}(\psi)_{z_0}=\mathcal{I}_+(\psi)_{z_0}$ for any $z_0\in D$. When $\psi(z_0)>-\infty$, it is clear that $\mathcal{I}(\psi)_{z_0}=\mathcal{I}_+(\psi)_{z_0}=\mathcal O_{z_0}$.  In the following part, we will only consider
the case that $\psi(z_0)=-\infty$ and $\psi$ is not identically $-\infty$ in any of the neighborhoods of $z_0$.

Firstly, we will prove that there is a plurisubharmonic function $\widetilde\psi$ whose polar set is included in a (closed) analytic set, satisfying $\mathcal{I}(\psi)_{z_0}=\mathcal{I}(\widetilde\psi)_{z_0}$ and $\mathcal{I}_+(\psi)_{z_0}=\mathcal{I}_+(\widetilde\psi)_{z_0}$. Then it suffices to prove $\mathcal{I}(\widetilde\psi)_{z_0}=\mathcal{I}_+(\widetilde\psi)_{z_0}$.

Choosing $\epsilon>0$, let $J_{\epsilon}:=\mathcal I(3(1+\epsilon)\psi)_{z_0}=(f_1,f_2,...,f_m)_{z_0}$ and $|J_{\epsilon}|^2:=\sum_{j=1}^m|f_j|^2$. By the definition of $\mathcal I(3(1+\epsilon)\psi)_{z_0}$, we have $\int_V|J_{\epsilon}|^2e^{-3(1+\epsilon)\psi}<+\infty$, where $V\ni z_0$ is a neighborhood of $z_0$ in $D$. Then we obtain that
\begin{equation}
	\label{eq:soc1}
	\begin{split}
		&\int_Ve^{-(1+\epsilon)\psi}-e^{-\max\{(1+\epsilon)\psi,\log|J_{\epsilon}|\}}\\
		\leq&\int_{\{\log|J_{\epsilon}|>(1+\epsilon)\psi\}\cap V}e^{-(1+\epsilon)\psi}|J_{\epsilon}|^2e^{-2(1+\epsilon)\psi}\\
		<&+\infty.
	\end{split}
\end{equation}
Let $\widetilde\psi=\max\{\psi,\frac{\log|f_1|}{1+\epsilon}\}$, it is clear that
\begin{equation}
	\label{eq:soc2}
	\int_Ve^{-(1+\epsilon)\psi}-e^{-(1+\epsilon)\widetilde\psi}\leq\int_Ve^{-(1+\epsilon)\psi}-e^{-\max\{(1+\epsilon)\psi,\log|J_{\epsilon}|\}}.
\end{equation}

We may of course assume that $|f_1|<1$ on $V$, then we obtain that $\widetilde\psi<0$ on $V$. For any $s\in[0,\epsilon]$, it follows from $e^{-(1+s)\psi}-e^{-(1+s)\widetilde\psi}\leq e^{-(1+\epsilon)\psi}-e^{-(1+\epsilon)\widetilde\psi}$, inequality \eqref{eq:soc1} and inequality \eqref{eq:soc2} that
\begin{equation}
	\label{eq:soc3}
		\int_Ve^{-(1+s)\psi}-e^{-(1+s)\widetilde\psi}
		\leq\int_Ve^{-(1+\epsilon)\psi}-e^{-(1+\epsilon)\widetilde\psi}
		<+\infty.	
\end{equation}
Combining inequality \eqref{eq:soc3} and $\psi\leq\widetilde\psi$, we obtain   $\mathcal{I}((1+s)\psi)_{z_0}=\mathcal{I}((1+s)\widetilde\psi)_{z_0}$ for any $s\in[0,\epsilon]$.

Next, we will prove $\mathcal{I}(\widetilde\psi)_{z_0}=\mathcal{I}_+(\widetilde\psi)_{z_0}$ by contradiction. If not, then there exists a holomorphic function $F$ near $z_0$ such that $(F,z_0)\in\mathcal{I}(\widetilde\psi)_{z_0}$ and $(F,z_0)\notin\mathcal{I}((1+s)\widetilde\psi)_{z_0}=\mathcal{I}_+(\widetilde\psi)_{z_0}$ for some $s\in(0,\epsilon]$.

Choose a small enough neighborhood of $z_0$ denoted also by $D$,  which satisfies that $D$ is a bounded pseudoconvex domain, $D \subset V$ and$$\int_D|F|^2e^{-\widetilde\psi}<+\infty.$$ Then we claim that there exist $C_0>0$ and $T>0$ such that $C_{F,(1+s)\widetilde\psi,t}(z_0)>C_0$ for any $t>T$, which implies that $C_{F,(1+s')\widetilde\psi,t}(z_0)>C_0$ for any $t>T$ and $s'\in (0,s]$. In fact, as $C_{F,(1+s)\widetilde\psi,t}(z_0)$ is increasing with respect to $t$, if the claim is wrong, then we have $C_{F,(1+s)\widetilde\psi,t}(z_0)=0$ for any $t>0$.

As $f_1$ is a holomorphic function near $z_0=(z_{0,1},z_{0,2},...,z_{0,n})$ and $f_1$ is not identically $0$, for any $v\in\mathbb C^n$ the Taylor series of $f_1$ at $z_0$ can be written as
$$f_1(z_0+tv)=\sum_{k=0}^{+\infty}\frac{1}{k!}t^kf^{k}(v),$$
where $f^{k}$ is a homogeneous polynomial of degree $k$ on $\mathbb C^n$ and $f^{k_0}$ is not identically 0 for   some $k_0$. Thus we may assume that $f^{k_0}(v) \not=0$, where $v=(1,0,...,0)$. Then there exists $r_1$ such that $f_1(w_1,z_{0,2},z_{0,3},...,z_{0,n})\not=0$ and $(w_1,z_{0,2},z_{0,3},...,z_{0,n})\in D$ when $0<|w_1-z_{0,1}|\leq r_1$,
which follows that there exist $r>0$ and $a>0$ such that $D(z_{0,1},r_1)\times \prod_{l=2}^nD(z_{0,l},r)\subset D$ where $D(y,r):=\{z\in\mathbb C:|z-y|<r\}$ and $|f_1(w)|>a$ for any $w\in\{w_1\in\mathbb C:\frac{r_1}{2}<|w_1-z_{0,1}|<r_1\}\times \prod_{l=2}^nD(z_{0,l},r)$ denoted by $W$.
Choosing $t_0>-\log a$, we have
\begin{equation}
	\label{eq:soc4}
	W\subset\{\log |f_1|\geq -t_0\}\cap D\subset \{(1+s)\widetilde\psi\geq-t_0\}\cap D.
\end{equation}
It follows from $C_{F,(1+s)\widetilde\psi,t_0}(z_0)=0$ that for any positive integer $j$ there exists $F_j\in\mathcal O(D)$ such that $(F_j-F,z_0)\in \mathcal I((1+s)\widetilde\psi)_{z_0}$ and
\begin{equation}
	\label{eq:soc5}
	\int_W|F_j|^2\leq\int_{\{(1+s)\widetilde\psi\geq-t_0\}\cap D}|F_j|^2<\frac{1}{j}.
\end{equation}
Combining inequality \eqref{eq:soc5} and Maximum principle,
we obtain that $\{F_j\}$ compactly convergent to $0$ on  $D(z_{0,1},r_1)\times \prod_{l=2}^nD(z_{0,l},r)$.
Note that Lemma \ref{closedness} implies that $(-F,z_0)\in \mathcal I((1+s)\widetilde\psi)_{z_0}$,
which contradicts $(F,z_0)\not\in \mathcal I((1+s)\widetilde\psi)_{z_0}$.
Then the claim holds.

As $\lim_{t\rightarrow+\infty}\int_{D_t}|F|^2e^{-\widetilde\psi}=\int_{D}|F|^2e^{-\widetilde\psi}<+\infty$, let $D_{t}=\{z\in D:\widetilde\psi\geq -t\}$, then We have
\begin{equation}
	\label{eq:soc6}
	\lim_{t\rightarrow+\infty}\lim_{s'\rightarrow0+}\frac{\int_{D_t}|F|^2e^{-(1+s')\widetilde\psi}}{C_{F,(1+s')\widetilde\psi,t}(z_0)}\leq\lim_{t\rightarrow+\infty}\frac{\int_{D_t}|F|^2e^{-\widetilde\psi}}{C_0}=\frac{\int_{D}|F|^2e^{-\widetilde\psi}}{C_0}<+\infty.
\end{equation}
But using Theorem \ref{thm: support function}, we obtain that
\begin{displaymath}
	\frac{\int_{D_t}|F|^2e^{-(1+s')\widetilde\psi}}{C_{F,(1+s')\widetilde\psi,t}(z_0)}
	\geq\frac{\int_{\{(1+s')\widetilde\psi\geq-t\}\cap D}|F|^2e^{-(1+s')\widetilde\psi}}{C_{F,(1+s')\widetilde\psi,t}(z_0)}
	\geq\frac{t}{1-e^{-t}},
\end{displaymath}
which contradicts inequality \eqref{eq:soc6}.
Then Corollary \ref{soc} has been proved.

\section{Appendix}

In this section, we present a concavity property of $G_{F,\psi,c}$ and the proof of Lemma \ref{l: estimates}.

\subsection{A concavity property of $G_{F,\psi,c}$}

Let $h(t)=\int_t^{+\infty}c(l)e^{-l}dl$ defined on $t\in[0,+\infty)$.
It is clear that $h(t)\in(0,\int_0^{+\infty}c(l)e^{-l}dl]$,
and $h(t)$ is strictly decreasing with respect to $t$.

Proposition \ref{p} shows that
$$G_{F,\psi,c}(t)\geq \frac{\int_t^{+\infty}c(l)e^{-l}dl}{\int_0^{+\infty}c(l)e^{-l}dl}G_{F,\psi,c}(0), $$ when $G_{F,\psi,c}(0)<+\infty$.
Let $r=h(t)$, then we obtain that
\begin{displaymath}
	\begin{split}
		G_{F,\psi,c}(h^{-1}(r))&\geq \frac{h(h^{-1}(r))}{\int_0^{+\infty}c(l)e^{-l}dl}G_{F,\psi,c}(0)\\
		&\geq\frac{r}{\int_0^{+\infty}c(l)e^{-l}dl}G_{F,\psi,c}\left(h^{-1}\left(\int_0^{+\infty}c(l)e^{-l}dl\right)\right).	\end{split}
\end{displaymath}
So it is natural to consider the concavity of $G_{F,\psi,c}(h^{-1}(r))$.

\begin{Proposition}
\label{p2}
Let $\psi$ be a negative plurisubharmonic function defined on a pseudoconvex domain $D\subset \mathbb C^n$, $F$ be a holomorphic function on $D$, $c\in P$, and $\psi(z_0)=-\infty$. If $G_{F,\psi,c}(0)<+\infty$, then $G_{F,\psi,c}(h^{-1}(r))$ is concave with respect to $r\in(0,\int_0^{+\infty}c(l)e^{-l}dl]$.
\end{Proposition}

We complete the proof of the Proposition \ref{p2} by the following two lemmas.

\begin{Lemma}
	\label{l:21}
	Under the assumption of Proposition \ref{p2}, for any $t_1\in[0,+\infty)$ and $t_0\in(0,+\infty)$, 	we have
	\begin{displaymath}
		\begin{split}
			G_{F,\psi,c}(t_1)-G_{F,\psi,c}(t_0+t_1)\leq &\frac{\int_{t_1}^{t_0+t_1}c(l)e^{-l}dl}{c(t_0+t_1)e^{-(t_0+t_1)}}\\
			&\times \liminf_{B\rightarrow0+}\frac{G_{F,\psi,c}(t_0+t_1)-G_{F,\psi,c}(t_0+t_1+B)}{B},		\end{split}
	\end{displaymath}
	i.e. $$\frac{G_{F,\psi,c}(t_1)-G_{F,\psi,c}(t_0+t_1)}{\int_{t_1}^{+\infty}c(l)e^{-l}dl-\int_{t_0+t_1}^{+\infty}c(l)e^{-l}dl}\leq\liminf_{B\rightarrow0+}\frac{G_{F,\psi,c}(t_0+t_1)-G_{F,\psi,c}(t_0+t_1+B)}{\int_{t_0+t_1}^{+\infty}c(l)e^{-l}dl-\int_{t_0+t_1+B}^{+\infty}c(l)e^{-l}dl}.$$
	
\end{Lemma}

\begin{proof}
	By replacing $D$ with $\{\psi<-t_1\}$, replacing $\psi$ with $\psi+t_1$ and replacing $c(t)$ by $c(t+t_1)$, it follows from Lemma \ref{l:5} that this lemma holds.
\end{proof}

As $G_{F,\psi,c}(h^{-1}(r))$ is lower semi-continuous (Lemma \ref{l:4}), then it follows from the following well-known property of concave functions (Lemma \ref{l:22}) that Lemma \ref{l:21} implies Proposition \ref{p2}.

\begin{Lemma}
\label{l:22} (see \cite{guan_sharp})
Let $a(r)$ be a lower semicontinuous function on $(0,R]$.
Then $a(r)$ is concave if and only if
$$\frac{a(r_{1})-a(r_{2})}{r_{1}-r_{2}}\leq \liminf_{r_{3}\to r_{2}-}\frac{a(r_{3})-a(r_{2})}{r_{3}-r_{2}}$$
holds for any $0< r_{2}<r_{1}\leq R$.
\end{Lemma}

\subsection{Proof of Lemma \ref{l: estimates}}
\label{sec:Lemma_L2}

The following remark shows that it suffices to consider the case of Lemma \ref{l: estimates} that $D$ is a strongly pseudoconvex domain, $\varphi$ and $\psi$ are plurisubharmonic functions on an open set $U$ containing $\bar{D}$ such that $\psi<0$ on $U$, and $F$ is a holomorphic function on $U\cap \{\psi<-t_0\}$ such that $\int_{D\cap \{\psi<-t_0\}}|F|^2<+\infty$.

In the following remark, we recall some standard steps (see e.g. \cite{siu96,guan-zhou13p,guan-zhou13ap}) to illustrate the above statement.
\begin{Remark}
\label{r4}
	It is well-known that there exist strongly pseudoconvex domains $D_{1}\Subset\cdots\Subset D_{j}\Subset D_{j+1}\Subset\cdots$ such that $\cup_{j=1}^{+\infty}D_j=D$.
	
	If inequality \eqref{eq:l1} holds on $D$, then we obtain a sequence of holomorphic functions $\widetilde F_{j}$ on $D_j$ such that
	\begin{equation}
		\label{ep:r41}
		\begin{split}
		&\int_{D_{j}}|\widetilde{F}_{j}-(1-b_{t_0,B}(\psi))F|^{2}e^{-\varphi+v_{t_0,B}(\psi)}c(-v_{t_0,B}(\psi))\\\leq
		&\int_{0}^{t_{0}+B}c(t)e^{-t}dt\int_{D_{j}}\frac{1}{B}\mathbb{I}_{\{-t_{0}-B<\psi<-t_{0}\}}|F|^{2}e^{-\varphi}\\\leq
		&C\int_{0}^{t_{0}+B}c(t)e^{-t}dt		
		\end{split}		
	\end{equation}
is bounded with respect to $j$. Note that for any given $j$, $e^{-\varphi+v_{t_0,B}(\psi)}c(-v_{t_0,B}(\psi))$ has a positive lower bound on $D_j$, then it follows that for any given $j$, $\int_{D_{j}}|\widetilde{F}_{j'}-(1-b_{t_0,B}(\psi))F|^{2}$ is bounded with respect to $j'>j$. As $$\int_{D_{j}}|(1-b_{t_0,B}(\psi))F|^{2}<\int_{D_j\cap\{\psi<-t_0\}}|F|^2<+\infty,$$ then we obtain that $\int_{D_j}|\widetilde F_{j'}|^2$ is bounded with respect to $j'>j$.

By diagonal method, there exists a sequence of  $\widetilde F_{j_k}$ uniformly convergent on any $\bar D_{j}$ 	to a holomorphic function on $D$ denoted by $\widetilde F$. Then it follows from inequality \eqref{ep:r41} and Fatou's Lemma that 	$$\int_{D_{j}}|\widetilde{F}-(1-b_{t_0,B}(\psi))F|^{2}e^{-\varphi+v_{t_0,B}(\psi)}c(-v_{t_0,B}(\psi))\leq C\int_{0}^{t_{0}+B}c(t)e^{-t}dt,$$
	then when $j$ goes to $+\infty$ we obtain Lemma \ref{l: estimates}. 	
	\end{Remark}

For the sake of completeness, we recall some lemmas on $L^2$ estimates for $\bar\partial$-equations, and $\bar\partial^*$ means the Hilbert adjoint operator of $\bar\partial$.

\begin{Lemma}
	\label{l23}
	(see \cite{siu96}, see also \cite{berndtsson})
	Let $\Omega\Subset\mathbb C^n$ be a domain with $C^{\infty}$ boundary $b\Omega$, $\Phi\in C^{\infty}(\bar\Omega)$. Let $\rho$ be a  $C^{\infty}$ defining  function for $\Omega$ such that $|d\rho|=1$ on $b\Omega$. Let $\eta$ be a smooth function on $\bar\Omega$. For any $(0,1)$-form $\alpha=\sum_{j=1}^{n}\alpha_{\bar j}d\bar z^{j}\in Dom_{\Omega}(\bar\partial^*)\cap C_{(0,1)}^{\infty}(\bar\Omega)$,
	\begin{equation}
		\label{eq:l231}
		\begin{split}
			&\int_{\Omega}\eta|\bar\partial_{\Phi}^*\alpha|^2e^{-\Phi}+\int_{\Omega}\eta|\bar\partial\alpha|^2e^{-\Phi}=\sum_{i,j=1}^n\int_{\Omega}\eta|\bar\partial_{i}\alpha_{\bar j}|^2e^{-\Phi}\\
			&+\sum_{i,j=1}^n\int_{b\Omega}\eta(\partial_i\bar\partial_j\rho)\alpha_{\bar i}\bar\alpha_{\bar j}e^{-\Phi}+\sum_{i,j=1}^n\int_{\Omega}\eta(\partial_i\bar\partial_j\Phi)\alpha_{\bar i}\bar\alpha_{\bar j}e^{-\Phi}\\
			&+\sum_{i,j=1}^n\int_{\Omega}-(\partial_i\bar\partial_j\eta)\alpha_{\bar i}\bar\alpha_{\bar j}e^{-\Phi}+2\mathrm{Re}(\bar\partial_{\Phi}^*\alpha,\alpha\llcorner(\bar\partial\eta)^{\sharp})_{\Omega,\Phi},
							\end{split}
	\end{equation}
	where $\alpha\llcorner(\bar\partial\eta)^{\sharp}=\sum_j\alpha_{\bar j}\partial_j\eta$.
	
		\end{Lemma}

The symbols and notations can be referred to \cite{guan-zhou13ap}. See also \cite{siu96}, \cite{siu00},or \cite{Straube}.

\begin{Lemma}
	\label{l24}
	(see \cite{berndtsson}, see also \cite{guan-zhou13ap}) Let $\Omega\Subset\mathbb C^n$ be a strictly pseudoconvex domain with $C^{\infty}$ boundary $b\Omega$ and $\Phi\in C(\bar\Omega)$. Let $\lambda$ be a $\bar\partial$ closed smooth form of bidegree $(n,1)$ on $\bar\Omega$. Assume the inequality
	$$|(\lambda,\alpha)_{\Omega,\Phi}|^2\leq C\int_{\Omega}|\bar\partial_{\Phi}^*\alpha|^2\frac{e^{-\Phi}}{\mu}<+\infty,$$
	where $\frac{1}{\mu}$ is an integrable positive function on $\Omega$ and $C$ is a constant, holds for all $(n,1)$-form $\alpha\in Dom_{\Omega}(\bar\partial^*)\cap Ker(\bar\partial)\cap C_{(n,1)}^{\infty}(\bar\Omega)$. Then there is a solution $u$ to the equation $\bar\partial u=\lambda$ such that
	$$\int_{\Omega}|u|^2\mu e^{-\Phi}\leq C.$$
\end{Lemma}

In the following part,  we will prove Lemma \ref{l: estimates}.

For the sake of completeness, let us recall some steps in the proof in \cite{guan-zhou13p} (see also \cite{guan-zhou13ap}, \cite{GZopen-effect}, \cite{guan_sharp}) with some slight modifications.

By Remark \ref{r4}, we can assume that $D\Subset\mathbb C^n$ is a strongly pseudoconvex domain (with smooth boundary), $\varphi$ and $\psi$ are plurisubharmonic functions on an open set $U$ containing $\bar{D}$ such that $\psi<0$ on $U$, and $F$ is a holomorphic function on $U\cap \{\psi<-t_0\}$ such that
\begin{equation}
	\label{eq:14}
	\int_{D\cap \{\psi<-t_0\}}|F|^2<+\infty.
\end{equation}
Then it follows from method of convolution (see e.g. \cite{demailly-book}) that there exist smooth plurisubharmonic functions $\psi_m$ and $\varphi_m$ on an open set $U\supset \bar D$  decreasing convergent to $\psi$ and $\varphi$ respectively, such that $\sup_m\sup_D\psi_m<0$ and $\sup_m\sup_D\varphi_m<+\infty$.

  \

\emph{Step 1: recall some Notations}

\

Let $\epsilon\in(0,\frac{1}{8}B)$.
Let $\{v_{\epsilon}\}_{\epsilon\in(0,\frac{1}{8}B)}$ be a family of smooth increasing convex functions on $\mathbb{R}$,
which are continuous functions on $\mathbb{R}\cup\{-\infty\}$, such that:

 $1)$ $v_{\epsilon}(t)=t$ for $t\geq-t_{0}-\epsilon$, $v_{\epsilon}(t)=constant$ for $t<-t_{0}-B+\epsilon$ and are pointwise convergent to $v_{t_0,B}$, when $\epsilon\to 0$;

 $2)$ $v''_{\epsilon}(t)$ are pointwise convergent to $\frac{1}{B}\mathbb{I}_{(-t_{0}-B,-t_{0})}$, when $\epsilon\to 0$,
 and $0\leq v''_{\epsilon}(t)\leq \frac{2}{B}\mathbb{I}_{(-t_{0}-B+\epsilon,-t_{0}-\epsilon)}$ for any $t\in \mathbb{R}$;

 $3)$ $v'_{\epsilon}(t)$ are pointwise convergent to $b_{t_0,B}(t)$ which is a continuous function on $\mathbb{R}\cup\{-\infty\}$, when $\epsilon\to 0$, and $0\leq v'_{\epsilon}(t)\leq1$ for any $t\in \mathbb{R}$.

One can construct the family $\{v_{\epsilon}\}_{\epsilon\in(0,\frac{1}{8}B)}$ by the setting
\begin{equation}
\label{eq:15}
v_{\epsilon}(t):=\int_{0}^{t}\left(\int_{-\infty}^{t_{1}}\left(\frac{1}{B-4\epsilon}
\mathbb{I}_{(-t_{0}-B+2\epsilon,-t_{0}-2\epsilon)}*\rho_{\frac{1}{4}\epsilon}\right)(s)ds\right)dt_{1},
\end{equation}
where $\rho_{\frac{1}{4}\epsilon}$ is the kernel of convolution satisfying $\mathrm{supp}(\rho_{\frac{1}{4}\epsilon})\subset (-\frac{1}{4}\epsilon,\frac{1}{4}\epsilon)$.
Then it follows that
$$v''_{\epsilon}(t)=\frac{1}{B-4\epsilon}\mathbb{I}_{(-t_{0}-B+2\epsilon,-t_{0}-2\epsilon)}*\rho_{\frac{1}{4}\epsilon}(t),$$
and
$$v'_{\epsilon}(t)=\int_{-\infty}^{t}\left(\frac{1}{B-4\epsilon}\mathbb{I}_{(-t_{0}-B+2\epsilon,-t_{0}-2\epsilon)}
*\rho_{\frac{1}{4}\epsilon}\right)(s)ds.$$
It is clear that $\lim_{\epsilon\to 0}v_{\epsilon}(t)=v_{t_0,B}(t)$, and $\lim_{\epsilon\to 0}v'_{\epsilon}(t)=b_{t_0,B}(t)$.

Let $\eta=s(-v_{\epsilon}(\psi_{m}))$ and $\phi=u(-v_{\epsilon}(\psi_{m}))$,
where $s\in C^{\infty}((0,+\infty))$ satisfies $s\geq0$ and $s'>0$, and
$u\in C^{\infty}((0,+\infty))$, satisfies $\lim_{t\to+\infty}u(t)$ exists, such that $u''s-s''>0$, and $s'-u's=1$.
It follows from $\sup_{m}\sup_{D}\psi_{m}<0$ that $\phi=u(-v_{\epsilon}(\psi_{m}))$ are uniformly bounded
on $D$ with respect to $m$ and $\epsilon$,
and $u(-v_{\epsilon}(\psi))$ are uniformly bounded
on $D$ with respect to $\epsilon$.
Let $\Phi=\phi+\varphi_{m'}$.

\

\emph{Step 2: Solving $\bar\partial$-equation with smooth polar function and smooth weight}

\

Now let $\alpha=\sum_{j=1}^n\alpha_{\bar j}d\bar z^j\in$Dom$_{D}(\bar\partial^*)\cap $Ker$(\bar\partial)\cap C_{(0,1)}^{\infty}(\bar D)$. It follows from Cauchy-Schwarz inequality that
\begin{equation}
\label{eq:17}
	\begin{split}
		2\mathrm{Re}(\bar\partial_{\Phi}^*\alpha,\alpha\llcorner(\bar\partial\eta)^{\sharp})_{D,\Phi}\geq &-\int_{D}g^{-1}|\bar\partial_{\Phi}^*\alpha|^2e^{-\Phi}\\&+\sum_{j,k=1}^n\int_{D}-g(\partial_j\eta)(\bar\partial_k\eta)\alpha_{\bar j}\bar\alpha_{\bar k}e^{-\Phi},
	\end{split}
\end{equation}
 where $g$ is a positive continuous function on $\bar D$.

Using Lemma \ref{l23} and inequality \eqref{eq:17}, since $s\geq0$, $\varphi_{m'}$ is a plurisubharmonic function on $\bar D$ and $D$ is a strongly pseudoconvex domain, we get
\begin{equation}
	\label{eq:18}
	\begin{split}
		&\int_{D}(\eta+g^{-1})|\bar\partial_{\Phi}^*\alpha|^2e^{-\Phi}\\
		\geq &\sum_{j,k=1}^n\int_D(-\partial_j\bar\partial_k\eta+\eta\partial_j\bar\partial_k\Phi-g(\partial_j\eta)(\bar\partial_k\eta))\alpha_{\bar j}\bar\alpha_{\bar k}e^{-\Phi}\\
		\geq &\sum_{j,k=1}^n\int_D(-\partial_j\bar\partial_k\eta+\eta\partial_j\bar\partial_k\phi-g(\partial_j\eta)(\bar\partial_k\eta))\alpha_{\bar j}\bar\alpha_{\bar k}e^{-\Phi}.	\end{split}
\end{equation}
We need some calculations to determine $g$.

We have
\begin{equation}
	\label{eq:19}
	\partial_j\bar\partial_k\eta=-s'(-v_{\epsilon}(\psi_m))\partial_j\bar\partial_k(v_{\epsilon}(\psi_m))+s''(-v_{\epsilon}(\psi_m))\partial_jv_{\epsilon}(\psi_m)\bar\partial_kv_{\epsilon}(\psi_m),
	\end{equation}
and
\begin{equation}
	\label{eq:20}
	\partial_j\bar\partial_k\phi=-u'(-v_{\epsilon}(\psi_m))\partial_j\bar\partial_k(v_{\epsilon}(\psi_m))+u''(-v_{\epsilon}(\psi_m))\partial_jv_{\epsilon}(\psi_m)\bar\partial_kv_{\epsilon}(\psi_m)
\end{equation}
for any $j,k(1\leq j,k\leq n)$.

Then we have
\begin{equation}
\label{eq:21}
\begin{split}
&\sum_{j,k=1}^n(-\partial_j\bar{\partial}_k\eta+\eta\partial_j\bar{\partial}_k\phi-g(\partial_j\eta)(\bar\partial_k\eta))\alpha_{\bar j}\bar\alpha_{\bar k}
\\=&(s'-su')\sum_{j,k=1}^n\partial_j\bar{\partial_k}v_{\epsilon}(\psi_{m})\alpha_{\bar j}\bar\alpha_{\bar k}\\
&+((u''s-s'')-gs'^{2})\sum_{j,k=1}^n\partial_j
(-v_{\epsilon}(\psi_{m}))\bar{\partial}_k(-v_{\epsilon}(\psi_{m}))\alpha_{\bar j}\bar\alpha_{\bar k}
\\=&(s'-su')\sum_{j,k=1}^n(v'_{\epsilon}(\psi_{m})\partial_j\bar{\partial}_k\psi_{m}+v''_{\epsilon}(\psi_{m})
\partial_j(\psi_{m})\bar{\partial}_k(\psi_{m}))\alpha_{\bar j}\bar\alpha_{\bar k}
\\&+((u''s-s'')-gs'^{2})\sum_{j,k=1}^n\partial_j
(-v_{\epsilon}(\psi_{m}))\bar{\partial}_k(-v_{\epsilon}(\psi_{m}))\alpha_{\bar j}\bar\alpha_{\bar k}
.
\end{split}
\end{equation}
We omit composite item $-v_{\epsilon}(\psi_{m})$ after $s'-su'$ and $(u''s-s'')-gs'^{2}$ in the above equalities.

Let $g=\frac{u''s-s''}{s'^2}(-v_{\epsilon}(\psi_m))$. It follows that $\eta+g^{-1}=\left(s+\frac{s'^2}{u''s-s''}\right)(-v_{\epsilon}(\psi_m))$.

As $v'_{\epsilon}\geq0$ and $s'-su'=1$, using inequality \eqref{eq:18}, we obtain that
\begin{equation}
	\label{eq:22}
	\int_{D}(\eta+g^{-1})|\bar\partial_{\Phi}^*\alpha|^2e^{-\Phi}\geq\int_D v''_{\epsilon}(\psi_m)|\alpha\llcorner(\bar\partial\psi_m)^{\sharp}|^2e^{-\Phi}.
\end{equation}

As $F$ is holomorphic on $\{\psi<-t_0\}$ and $\mathrm{supp}(v''_{\epsilon}(\psi_m))\subset\{\psi<-t_0\}$, then $\lambda:=\bar\partial((1-v'_{\epsilon}(\psi_m))F)$ is well-defined and smooth on $D$. By the definition of contraction, Cauchy-Schwarz inequality and inequality \eqref{eq:22}, it follows that
\begin{equation}
	\label{eq:23}
	\begin{split}
		|(\lambda,\alpha)_{D,\Phi}|^2=&|(v''_{\epsilon}(\psi_m)\bar\partial\psi_mF,\alpha)_{D,\Phi}|^2\\
		=&|(v''_{\epsilon}(\psi_m)F,\alpha\llcorner(\bar\partial\psi_m)^{\sharp})_{D,\Phi}|^2\\
		\leq&\left(\int_Dv''_{\epsilon}(\psi_m)|F|^2e^{-\Phi}\right)\left(\int_Dv''_{\epsilon}(\psi_m)|\alpha\llcorner(\bar\partial\psi_m)^{\sharp}|^2e^{-\Phi}\right)\\
		\leq&\left(\int_Dv''_{\epsilon}(\psi_m)|F|^2e^{-\Phi}\right)\left(\int_D(\eta+g^{-1})|\bar\partial_{\Phi}^*\alpha|^2e^{-\Phi}\right).
	\end{split}
\end{equation}

Let $\mu:=(\eta+g^{-1})^{-1}$. By Lemma \ref{l24}, then we have locally $L^1$ function $u_{m,m',\epsilon}$ on $D$ such that $\bar\partial u_{m,m',\epsilon}=\lambda$, and
\begin{equation}
	\label{eq:24}
		\int_{D}|u_{m,m',\epsilon}|^2(\eta+g^{-1})^{-1}e^{-\Phi}\leq\int_Dv''_{\epsilon}(\psi_m)|F|^2e^{-\Phi}.
\end{equation}

Assume that we can choose $\eta$ and $\phi$ such that $e^{v_{\epsilon}(\psi_{m})}e^{\phi}c(-v_{\epsilon}(\psi_{m}))=(\eta+g^{-1})^{-1}$.
Then inequality \eqref{eq:24} becomes
\begin{equation}
 \label{eq:25}
 \begin{split}
 &\int_{D}|u_{m,m',\epsilon}|^{2}e^{v_{\epsilon}(\psi_{m})-\varphi_{m'}}c(-v_{\epsilon}(\psi_{m}))
  \leq\int_{D}v''_{\epsilon}(\psi_{m})| F|^2e^{-\phi-\varphi_{m'}}.
  \end{split}
\end{equation}

Let $F_{m,m',\epsilon}:=-u_{m,m',\epsilon}+(1-v'_{\epsilon}(\psi_{m})){F}$. It is clear that $F_{m,m',\epsilon}$ is holomorphic on $D$.
Then inequality \eqref{eq:25} becomes
\begin{equation}
 \label{eq:26}
 \begin{split}
 &\int_{D}|F_{m,m',\epsilon}-(1-v'_{\epsilon}(\psi_{m})){F}|^{2}e^{v_{\epsilon}(\psi_{m})-\varphi_{m'}}c(-v_{\epsilon}(\psi_{m}))
  \\&\leq\int_{D}(v''_{\epsilon}(\psi_{m}))| F|^2e^{-\phi-\varphi_{m'}}.
  \end{split}
\end{equation}

\

\emph{Step 3: Singular polar function and smooth weight}

\

As $\sup_{m,\epsilon}|\phi|=\sup_{m,\epsilon}|u(-v_{\epsilon}(\psi_{m}))|<+\infty$ and $\varphi_{m'}$ is continuous on $\bar{D}$,
then $\sup_{m,\epsilon}e^{-\phi-\varphi_{m'}}<+\infty$.
Note that
$$v''_{\epsilon}(\psi_{m})| F|^2e^{-\phi-\varphi_{m'}}\leq\frac{2}{B}\mathbb{I}_{\{\psi<-t_{0}\}}| F|^{2}\sup_{m,\epsilon}e^{-\phi-\varphi_{m'}}$$
on $D$,
then it follows from inequality \eqref{eq:14} and the dominated convergence theorem that
\begin{equation}
\label{eq:27}
 \lim_{m\to+\infty}\int_{D}v''_{\epsilon}(\psi_{m})| F|^2e^{-\phi-\varphi_{m'}}=
\int_{D}v''_{\epsilon}(\psi)| F|^2e^{-u(-v_{\epsilon}(\psi))-\varphi_{m'}}.
\end{equation}

Note that $\inf_{m}\inf_{D}e^{v_{\epsilon}(\psi_{m})-\varphi_{m'}}c(-v_{\epsilon}(\psi_{m}))>0$,
then it follows from inequality \eqref{eq:26}
that $\sup_{m}\int_{D}|F_{m,m',\epsilon}-(1-v'_{\epsilon}(\psi_{m})){F}|^{2}<+\infty$.
Note that
\begin{equation}
\label{eq:28}
|(1-v'_{\epsilon}(\psi_{m}))F|\leq |\mathbb{I}_{\{\psi<-t_{0}\}}F|,
\end{equation}
then it follows from inequality \eqref{eq:14}
that $\sup_{m}\int_{D}|F_{m,m',\epsilon}|^{2}<+\infty$,
which implies that there exists a subsequence of $\{F_{m,m',\epsilon}\}_{m\in \mathbb{N}}$
(also denoted by $F_{m,m',\epsilon}$) compactly convergent to a holomorphic function $F_{m',\epsilon}$ on $D$.

Note that $e^{v_{\epsilon}(\psi_{m})-\varphi_{m'}}c(-v_{\epsilon}(\psi_{m}))$ are uniformly bounded on $D$ with respect to $m$,
then it follows from
$|F_{m,m',\epsilon}-(1-v'_{\epsilon}(\psi_{m})){F}|^{2}\leq
 2(|F_{m,m',\epsilon}|^{2}+ |(1-v'_{\epsilon}(\psi_{m})){F}|^{2})
\leq  2(|F_{m,m',\epsilon}|^{2}+  |\mathbb{I}_{\{\psi<-t_{0}\}}F^{2}|)$
and the dominated convergence theorem that
\begin{equation}
 \label{eq:29}
 \begin{split}
 \lim_{m\to+\infty}&\int_{K}|F_{m,m',\epsilon}-(1-v'_{\epsilon}(\psi_{m})){F}|^{2}e^{v_{\epsilon}(\psi_{m})-\varphi_{m'}}c(-v_{\epsilon}(\psi_{m}))
  \\=&\int_{K}|F_{m',\epsilon}-(1-v'_{\epsilon}(\psi)){F}|^{2}e^{v_{\epsilon}(\psi)-\varphi_{m'}}c(-v_{\epsilon}(\psi))
  \end{split}
\end{equation}
holds for any compact subset $K$ on $D$.
Combining with inequality \eqref{eq:26}, equality \eqref{eq:27} and equality \eqref{eq:29},
one can obtain that
\begin{equation}
 \label{eq:30}
 \begin{split}
&\int_{K}|F_{m',\epsilon}-(1-v'_{\epsilon}(\psi)){F}|^{2}e^{v_{\epsilon}(\psi)-\varphi_{m'}}c(-v_{\epsilon}(\psi))
\\&\leq
\int_{D}v''_{\epsilon}(\psi)| F|^2e^{-u(-v_{\epsilon}(\psi))-\varphi_{m'}},
\end{split}
\end{equation}
which implies
\begin{equation}
 \label{eq:31}
 \begin{split}
&\int_{D}|F_{m',\epsilon}-(1-v'_{\epsilon}(\psi)){F}|^{2}e^{v_{\epsilon}(\psi)-\varphi_{m'}}c(-v_{\epsilon}(\psi))
\\&\leq
\int_{D}v''_{\epsilon}(\psi)| F|^2e^{-u(-v_{\epsilon}(\psi))-\varphi_{m'}}.
\end{split}
\end{equation}

\

\emph{Step 4: Nonsmooth cut-off function}

\

Note that
$\sup_{\epsilon}\sup_{D}e^{-u(-v_{\epsilon}(\psi))-\varphi_{m'}}<+\infty,$
and
$$v''_{\epsilon}(\psi)| F|^2e^{-u(-v_{\epsilon}(\psi))-\varphi_{m'}}\leq
\frac{2}{B}\mathbb{I}_{\{-t_{0}-B<\psi<-t_{0}\}}| F|^2\sup_{\epsilon}\sup_{D}e^{-u(-v_{\epsilon}(\psi))-\varphi_{m'}},$$
then it follows from inequality \eqref{eq:14} and the dominated convergence theorem that
\begin{equation}
\label{eq:32}
\begin{split}
&\lim_{\epsilon\to0}\int_{D}v''_{\epsilon}(\psi)| F|^2e^{-u(-v_{\epsilon}(\psi))-\varphi_{m'}}
\\=&\int_{D}\frac{1}{B}\mathbb{I}_{\{-t_{0}-B<\psi<-t_{0}\}}|F|^2e^{-u(-v_{t_0,B}(\psi))-\varphi_{m'}}
\\\leq&\left(\sup_{D}e^{-u(-v_{t_0,B}(\psi))}\right)\int_{D}\frac{1}{B}\mathbb{I}_{\{-t_{0}-B<\psi<-t_{0}\}}|F|^2e^{-\varphi_{m'}}<+\infty.
\end{split}
\end{equation}

Note that $\inf_{\epsilon}\inf_{D}e^{v_{\epsilon}(\psi)-\varphi_{m'}}c(-v_{\epsilon}(\psi))>0$,
then it follows from inequality \eqref{eq:31} and \eqref{eq:32} that
$\sup_{\epsilon}\int_{D}|F_{m',\epsilon}-(1-v'_{\epsilon}(\psi)){F}|^{2}<+\infty.$
Combining with
\begin{equation}
\label{eq:33}
\sup_{\epsilon}\int_{D}|(1-v'_{\epsilon}(\psi)){F}|^{2}\leq\int_{D}\mathbb{I}_{\{\psi<-t_{0}\}}|F^{2}|<+\infty,
\end{equation}
one can obtain that $\sup_{\epsilon}\int_{D}|F_{m',\epsilon}|^{2}<+\infty$,
which implies that
there exists a subsequence of $\{F_{m',\epsilon}\}_{\epsilon>0}$ (also denoted by $\{F_{m',\epsilon}\}_{\epsilon>0}$)
compactly convergent to a holomorphic function $F_{m'}$ on $D$.

Note that $\sup_{\epsilon}\sup_{D}e^{v_{\epsilon}(\psi)-\varphi_{m'}}c(-v_{\epsilon}(\psi))<+\infty$ and
$|F_{m',\epsilon}-(1-v'_{\epsilon}(\psi)){F}|^{2}\leq 2(|F_{m',\epsilon}|^{2}+|\mathbb{I}_{\{\psi<-t_{0}\}}F|^{2})$,
then it follows from inequality \eqref{eq:33} and the dominated convergence theorem on any given $K\Subset D$ that
\begin{equation}
\label{eq:34}
\begin{split}
&\lim_{\epsilon\to0}\int_{K}|F_{m',\epsilon}-(1-v'_{\epsilon}(\psi)){F}|^{2}e^{v_{\epsilon}(\psi)-\varphi_{m'}}c(-v_{\epsilon}(\psi))
\\=&\int_{K}|F_{m'}-(1-b_{t_0,B}(\psi)){F}|^{2}e^{v_{t_0,B}(\psi)-\varphi_{m'}}c(-v_{t_0,B}(\psi)).
\end{split}
\end{equation}
Combining with inequality \eqref{eq:31}, inequality \eqref{eq:32} and equality \eqref{eq:34}, one can obtain that
\begin{equation}
\label{eq:35}
\begin{split}
&\int_{K}|F_{m'}-(1-b_{t_0,B}(\psi)){F}|^{2}e^{v_{t_0,B}(\psi)-\varphi_{m'}}c(-v_{t_0,B}(\psi))
\\\leq&\left(\sup_{D}e^{-u(-v_{t_0,B}(\psi))}\right)\int_{D}\frac{1}{B}\mathbb{I}_{\{-t_{0}-B<\psi<-t_{0}\}}|F|^2e^{-\varphi_{m'}}
\end{split}
\end{equation}
which implies
\begin{equation}
\label{eq:36}
\begin{split}
&\int_{D}|F_{m'}-(1-b_{t_0,B}(\psi)){F}|^{2}e^{v_{t_0,B}(\psi)-\varphi_{m'}}c(-v_{t_0,B}(\psi))
\\\leq&\left(\sup_{D}e^{-u(-v_{t_0,B}(\psi))}\right)\int_{D}\frac{1}{B}\mathbb{I}_{\{-t_{0}-B<\psi<-t_{0}\}}|F|^2e^{-\varphi_{m'}}.
\end{split}
\end{equation}

\

\emph{Step 5: Singular weight}

\

Note that
\begin{equation}
\label{eq:37}
\int_{D}\frac{1}{B}\mathbb{I}_{\{-t_{0}-B<\psi<-t_{0}\}}|F|^2e^{-\varphi_{m'}}\leq\int_{D}\frac{1}{B}\mathbb{I}_{\{-t_{0}-B<\psi<-t_{0}\}}|F|^{2}e^{-\varphi}<+\infty,
\end{equation}
and $\sup_{D}e^{-u(-v_{t_0,B}(\psi))}<+\infty$,
then it follows from \eqref{eq:36} that
$$\sup_{m'}\int_{D}|F_{m'}-(1-b_{t_0,B}(\psi)){F}|^{2}e^{v_{t_0,B}(\psi)-\varphi_{m'}}c(-v_{t_0,B}(\psi))<+\infty.$$
Combining with $\inf_{m'}\inf_{D}e^{v_{t_0,B}(\psi)-\varphi_{m'}}c(-v_{t_0,B}(\psi))>0$,
one can obtain that
$$\sup_{m'}\int_{D}|F_{m'}-(1-b_{t_0,B}(\psi)){F}|^{2}<+\infty.$$
Note that
\begin{equation}
\label{eq:38}
\int_{D}|(1-b_{t_0,B}(\psi)){F}|^{2}\leq\int_{D}|\mathbb{I}_{\{\psi<-t_{0}\}}F|^{2} <+\infty.
\end{equation}
Then $\sup_{m'}\int_{D}|F_{m'}|^{2}<+\infty$,
which implies that there exists a compactly convergent subsequence of $\{F_{m'}\}$ (also denoted by $\{F_{m'}\}$),
which is convergent a holomorphic function $\tilde{F}$ on $D$.

Note that $\sup_{D}e^{v_{t_0,B}(\psi)-\varphi_{m'}}c(-v_{t_0,B}(\psi))<+\infty$,
then it follows from inequality \eqref{eq:38} and the
dominated convergence theorem on any given compact subset $K$ of $D$ that
\begin{equation}
\label{eq:39}
\begin{split}
&\lim_{m''\to+\infty}\int_{K}|F_{m''}-(1-b_{t_0,B}(\psi)){F}|^{2}e^{v_{t_0,B}(\psi)-\varphi_{m'}}c(-v_{t_0,B}(\psi))
\\=&\int_{K}|\tilde{F}-(1-b_{t_0,B}(\psi)){F}|^{2}e^{v_{t_0,B}(\psi)-\varphi_{m'}}c(-v_{t_0,B}(\psi)).
\end{split}
\end{equation}
Note that for any $m''\geq m'$, $\varphi_{m'}\geq\varphi_{m''}$ holds,
then it follows from inequality \eqref{eq:36} and inequality \eqref{eq:37}
that
\begin{equation}
\label{eq:40}
\begin{split}
&\lim_{m''\to+\infty}\int_{K}|F_{m''}-(1-b_{t_0,B}(\psi)){F}|^{2}e^{v_{t_0,B}(\psi)-\varphi_{m'}}c(-v_{t_0,B}(\psi))
\\\leq&
\limsup_{m''\to+\infty}\int_{K}|F_{m''}-(1-b_{t_0,B}(\psi)){F}|^{2}e^{v_{t_0,B}(\psi)-\varphi_{m''}}c(-v_{t_0,B}(\psi))
\\\leq&
\limsup_{m''\to+\infty}\left(\sup_{D}e^{-u(-v_{t_0,B}(\psi))}\right)\int_{D}\frac{1}{B}\mathbb{I}_{\{-t_{0}-B<\psi<-t_{0}\}}|F|^2e^{-\varphi_{m''}}
\\\leq&
\left(\sup_{D}e^{-u(-v_{t_0,B}(\psi))}\right)C<+\infty.
\end{split}
\end{equation}
Combining with equality \eqref{eq:39},
one can obtain that
$$\int_{K}|\tilde{F}-(1-b_{t_0,B}(\psi)){F}|^{2}e^{v_{t_0,B}(\psi)-\varphi_{m'}}c(-v_{t_0,B}(\psi))\leq\left(\sup_{D}e^{-u(-v_{t_0,B}(\psi))}\right)C,$$
for any compact subset $K$ of $D$,
which implies
$$\int_{D}|\tilde{F}-(1-b_{t_0,B}(\psi)){F}|^{2}e^{v_{t_0,B}(\psi)-\varphi_{m'}}c(-v_{t_0,B}(\psi))\leq\left(\sup_{D}e^{-u(-v_{t_0,B}(\psi))}\right)C.$$
When $m'\to+\infty$,
it follows from the monotone convergence theorem that
\begin{equation}
\label{eq:41}
\begin{split}
\int_{D}|\tilde{F}-(1-b_{t_0,B}(\psi)){F}|^{2}e^{v_{t_0,B}(\psi)-\varphi}c(-v_{t_0,B}(\psi))\leq\left(\sup_{D}e^{-u(-v_{t_0,B}(\psi))}\right)C.
\end{split}
\end{equation}

\

\emph{Step 6: ODE system}

\

It suffices to find $\eta$ and $\phi$ such that
$\eta+g^{-1}=e^{-v_{\epsilon}(\psi_{m})}e^{-\phi}\frac{1}{c(-v_{\epsilon}(\psi_{m}))}$  on $D$ and $s'-u's=1$.
As $\eta=s(-v_{\epsilon}(\psi_{m}))$ and $\phi=u(-v_{\epsilon}(\psi_{m}))$,
we have $(\eta+g^{-1}) e^{v_{\epsilon}(\psi_{m})}e^{\phi}=\left(\left(s+\frac{s'^{2}}{u''s-s''}\right)e^{-t}e^{u}\right)\circ(-v_{\epsilon}(\psi_{m}))$.

Summarizing the above discussion about $s$ and $u$, we are naturally led to a
system of ODEs (see \cite{guan-zhou12,guan-zhou13p,guan-zhou13ap,GZopen-effect}):
\begin{equation}
\label{eq:42}
\begin{split}
&1).\,\,\left(s+\frac{s'^{2}}{u''s-s''}\right)e^{u-t}=\frac{1}{c(t)}, \\
&2).\,\,s'-su'=1,
\end{split}
\end{equation}
where $t\in(0,+\infty)$.

It is not hard to solve the ODE system \eqref{eq:42} and get $u(t)=-\log\left(\int_{0}^{t}c(t_{1})e^{-t_{1}}dt_{1}\right)$ and
$s(t)=\frac{\int_{0}^{t}\left(\int_{0}^{t_{2}}c(t_{1})e^{-t_{1}}dt_{1}\right)dt_{2}}{\int_{0}^{t}c(t_{1})e^{-t_{1}}dt_{1}}$
(see \cite{guan-zhou13ap}).
It follows that $s\in C^{\infty}((0,+\infty))$ satisfies $s>0$ and $s'>0$, $\lim_{t\to+\infty}u(t)=-\log\left(\int_{0}^{+\infty}c(t_{1})e^{-t_{1}}dt_{1}\right)$ and
$u\in C^{\infty}((0,+\infty))$ satisfies $u''s-s''>0$.

As $u(t)=-\log\left(\int_{0}^{t}c(t_{1})e^{-t_{1}}dt_{1}\right)$ is decreasing with respect to $t$,
then it follows from $0\geq v(t)\geq\max\{t,-t_{0}-B_{0}\}\geq -t_{0}-B_{0}$ for any $t\leq0$
that
\begin{equation}
\begin{split}
\sup_{D}e^{-u(-v_{t_0,B}(\psi))}
\leq\sup_{t\in(0,t_{0}+B]}e^{-u(t)}
=\int_{0}^{t_{0}+B}c(t_{1})e^{-t_{1}}dt_{1},
\end{split}
\end{equation}
therefore we are done.
Thus we have proved Lemma \ref{l: estimates}.

\vspace{.1in} {\em Acknowledgements}.
The first named author was supported by NSFC-11825101, NSFC-11522101 and NSFC-11431013.

\bibliographystyle{references}
\bibliography{xbib}

\end{document}